\documentclass[final, a4paper, 9pt]{article}

\usepackage[scale=0.768]{geometry}
\usepackage[english]{babel}

\usepackage[mathscr]{euscript}
\usepackage{color}
\usepackage{fancyhdr}
\usepackage{bm}
\usepackage{hyperref}
\usepackage{cite}
\usepackage{showkeys}
\usepackage{subfigure}
\usepackage{float}

\usepackage{amsfonts}
\usepackage{amsmath}
\usepackage{amssymb}
\usepackage{amscd}
\usepackage{amsthm}

\usepackage{accents}
\usepackage{graphicx}
\usepackage{array}

\usepackage{algorithm}
\usepackage{algorithmicx}
\usepackage{algpseudocode}
\newcolumntype{P}[1]{>{\centering\arraybackslash}p{#1}}

\newtheorem{theorem}{Theorem}
\newtheorem{definition}[theorem]{Definition}
\newtheorem{lemma}[theorem]{Lemma}
\newtheorem{corollary}[theorem]{Corollary}

\newtheorem{remark}[theorem]{Remark}
\newtheorem{assumption}[theorem]{Assumption}

\newcommand*{\N}{\ensuremath{\mathbb{N}}}
\newcommand*{\Z}{\ensuremath{\mathbb{Z}}}
\newcommand*{\R}{\ensuremath{\mathbb{R}}}
\newcommand*{\C}{\ensuremath{\mathbb{C}}}
\renewcommand{\i}{\mathrm{i}}
\renewcommand{\phi}{\varphi}
\renewcommand{\rho}{{\varrho}}
\renewcommand{\epsilon}{{\varepsilon}}

\renewcommand{\d}[1]{\,\mathrm{d}#1 \,}

\newcommand{\J}{\mathcal{J}} 

\newcommand{\0}{{0}} 

\newcommand{\T}{{\mathcal{T}}}

\newcommand{\B}{{\mathcal{B}}}

\newcommand{\A}{{\mathcal{A}}}

\renewcommand{\B}{{\mathcal{B}}}
\newcommand{\K}{{\mathcal{K}}}

\newcommand{\p}{{per}}
\renewcommand{\L}{\mathcal{L}} 
\renewcommand{\Re}{\mathrm{Re}\,}

\newcommand{\M}{{\mathcal{M}}}

\newlength{\dhatheight}

\usepackage{color}
\definecolor{xl}{rgb}{0.8,0.2,0.3}

\usepackage{amsmath}
\usepackage{amsfonts}
\usepackage{amssymb}
\begin{document}
	
	\sloppy\title{High order complex contour discretization methods to simulate scattering problems in locally perturbed periodic waveguides}
\author{
Ruming Zhang\thanks{Institute of Applied and Numerical mathematics, Karlsruhe Institute of Technology, Karlsruhe, Germany
; \texttt{ruming.zhang@kit.edu}. }}
\date{}
\maketitle

\begin{abstract}
    In this paper, two high order  {complex contour discretization} methods are proposed to simulate wave propagation in locally perturbed periodic closed waveguides. As is well known the problem is not always uniquely solvable due to the existence of guided modes. The limiting absorption principle is a standard way to get the unique physical solution. Both methods are based on the Floquet-Bloch transform which transforms the original problem to an equivalent family of cell problems.  {The first method, which is designed based on a complex contour integral of the inverse Floquet-Bloch transform, is called the CCI method. The second method, which comes from an explicit definition of the radiation condition, is called the decomposition method. Due to the local perturbation, the family of cell problems are coupled with respect to the Floquet parameter and the computational complexity becomes much larger.  {To this end, high order  methods to discretize the complex contours are developed to have better performances.} Finally we  {give} the convergence results  {which we confirm with} numerical examples.}
    \\

\end{abstract}
	
\section{Introduction}
Periodic structures are widely used in applications such as photonic crystals, for details we refer to \cite{Sakoda2001,Johns2002,Joann1995}. This topic also attracts the interests of many mathematicians and we refer to \cite{Kuchm2001,Kuchm2003,Lechl2011} for the studies from mathematical point of view.
It is well known that this kind of problems is challenging due to existence of guided modes. To obtain the unique physical solution, the {\em Limiting Absorption Principle (LAP)} is a standard process. The LAP process is to define the physical solution by the limit of unique solutions with absorption, as the absorption parameter tends to zero. For simplicity, in this paper we call the solution from the LAP process an LAP solution. Significant progresses have been made in the past few years in the study of this kind of  {problem}, from both theoretical and numerical point of view. For example,  in \cite{Hoang2011} with an analysis on the resolvent of the differential operator, a radiation condition was given for LAP solution in a periodic half guide, and in \cite{Fliss2015} the authors gave the radiation condition as well as  semi-analytic representations for LAP solutions in the full guide. On the other hand, with the singular perturbation theory (see \cite{Colto1992}), the radiation condition is given by authors in \cite{Kirsc2017a}. With this method, radiation conditions for LAP solutions are also developed for periodic layers in 2D spaces and periodic open tubes in 3D spaces, see \cite{Kirsc2017a,Kirsc2019a,Kirsc2019b}. Besides the LAP, a Kondrat'ev's weighted spaces based method was adopted by  S. A. Nazarov in \cite{Nazar1982} and further works were carried out by him and his collaborators in \cite{Nazar2014,Nazar1994,Nazar1990,Nazar1991,Nazar2013}. On the other hand, numerical methods are also developed to compute the LAP solutions. For example, in \cite{Joly2006} an algorithm was proposed to compute exact Dirichlet-to-Neumann maps from the LAP process and this method was extended to periodic structures with local perturbation \cite{Fliss2009,Fliss2009a,Fliss2012,Fliss2013}. A method based on the doubling recursive procedure with an extrapolation technique was also developed to compute the DtN maps, see \cite{Ehrhardt2009,Ehrhardt2009a,Sun2009}. With a decomposition of  Bloch waves, a numerical method was proposed for waveguides with different refractive indexes on both directions in  \cite{Dohna2018}.

In this paper, we develop two  high order  {complex contour discretization} methods to simulate wave scattered by local perturbations embedded in periodic closed waveguides. For both methods, the Floquet-Bloch transform is the key to transform the problem defined in 2D unbounded domain to an equivalent coupled family of cell problems. The idea comes from some older papers of the author and A. Lechleiter for (locally perturbed) periodic surfaces see \cite{Lechl2016,Lechl2016a,Lechl2016b,Lechl2017,Zhang2017e}.
Compared to the locally perturbed periodic waveguides, the surface scattering problems are always uniquely solvable thus the analysis is relatively easier. For the problems discussed in this paper, we need to introduce a radiation condition to describe the LAP solutions.  The first (complex contour integral/CCI) method  is based on the author's previous work on purely periodic waveguides from both theoretical (see \cite{Zhang2019a}) and numerical (see \cite{Zhang2019b}) point of view. The second (decomposition) method comes from an explicit  {characterization} of the radiation condition proposed in \cite{Nazar1994,Kirsc2017a}. For both methods, analytic formulations for LAP solutions are given as extensions of nonperturbed cases. Due to the local perturbation, the whole system is coupled with respect to the Floquet parameter thus it is a problem defined in 3D. Thus the computational complexity  {is} much larger than  {for} periodic problems, where the system is not coupled. Based on different types of singularities, we develop different high order algorithms for  both formulations. Finally, we show that  {both numerical schemes} converge super-algebraically in the domain of the Floquet parameters.

The remaining part of this paper is organized as follows. In the second section, the mathematical model is given and two explicit formulations via the Floquet-Bloch transform are given in the third section. In Section 4, different numerical schemes are developed, and numerical examples are shown in Section 5.

\section{Mathematical model}

Let the closed waveguide  $\Omega:=\R\times(0,1)$ with the upper and lower boundaries  
$\Sigma_-:=\R\times\{0\},\,\Sigma_+:=\R\times\{1\}$. The scattering problem in $\Omega$ is described by 
the following equations:
\begin{equation}\label{eq:waveguide}
\Delta u+k^2(n+q)u=f \text{ in }\Omega;\quad \frac{\partial u}{\partial x_2}=0\text{ on }\Sigma_\pm.
\end{equation}
Here $n$ is $1$-periodic in $x_1$-direction, both $f$ and $q$ are compactly supported. Moreover both $n$ and $n+q$ are strictly positive, i.e., there is a constant $c>0$ such that 
\[
n(x)\geq c>0;\quad n(x)+q(x)\geq c>0\quad\text{ for all }x\in\Omega.
\]
For convenience, we define the following periodicity cells and their boundaries:
\begin{align*}
&\Omega_j:=\left(j-\frac{1}{2},j+\frac{1}{2}\right)\times(0,1);\quad \Gamma_j=\left\{j-\frac{1}{2}\right\}\times(0,1);\\
& \Sigma_-^j=\left(j-\frac{1}{2},j+\frac{1}{2}\right)\times\{0\};\quad\Sigma_+^j=\left(j-\frac{1}{2},j+\frac{1}{2}\right)\times\{1\}.
\end{align*}
Thus $\partial\Omega_j=\Gamma_j\cup\Sigma_-^j\cup\Gamma_{j+1}\cup\Sigma_+^j$. For simplicity, we assume that both $f$ and $q$ are supported in $\Omega_0$.
 For visualization of the locally perturbed periodic waveguide we refer to Figure \ref{waveguide}. 

\begin{figure}[ht]
	\centering
	\includegraphics[width=12cm,height=2.25cm]{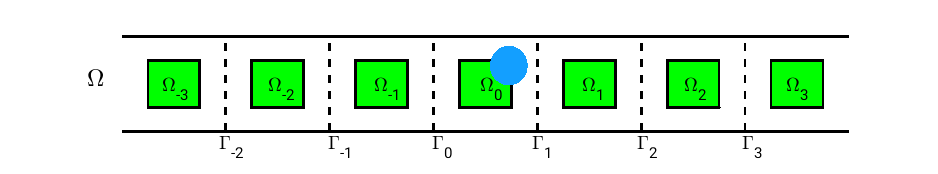}
	\caption{Periodic waveguide with local perturbation (blue disk).}
	\label{waveguide}
\end{figure}

As is well known, the  problem \eqref{eq:waveguide} is not always uniquely solvable for $k>0$. To carry out the LAP, we first consider the damped problem, by replacing $k^2$ with $k^2+\i\epsilon$ where $\epsilon>0$. 
It is well known that the damped problem is uniquely solvable in $H^1(\Omega)$. Then  the limit of solution when $\epsilon\rightarrow 0^+$  is set to be the physical solution, which is called an LAP solution in this paper. From \cite{Fliss2015,Zhang2019a,Kirsc2017a},
the radiation condition for LAP solutions in periodic waveguides have been described in different forms, and the forms are actually equivalent. 

\begin{definition}[Radiation Condition]\label{def:rc0}
	Suppose for the positive valued periodic refractive index $n$ and  wavenumber $k>0$, there are no standing waves. Then an LAP solution for \eqref{eq:waveguide} satisfies the following radiation conditions:
\begin{eqnarray*}
&& u(x)=u_+(x)+\sum_{\ell\in L_+}a_\ell^+\phi_\ell^+(x),\quad x_1>1/2;\\
&& u(x)=u_-(x)+\sum_{\ell\in L_-}a_\ell^-\phi_\ell^-(x),\quad x_1<-1/2;
\end{eqnarray*}
where $u_+$ ($u_-$)  decays exponentially when $x_1\rightarrow+\infty(-\infty)$, $\phi_\ell^+$ ($\phi_\ell^-$) are propagating modes  {traveling} to the right (left). $L_+$ ($L_-$) is the finite set of indexes for left (right) propagating modes and $a_\ell^\pm\in\C$ are coefficients.
\end{definition}

For definitions of standing waves and propagating modes we refer to Section 3.1 for details. The explicit formulation of LAP solutions plays a crucial role in development of numerical methods. In this paper, we show two ways to  develop  different numerical schemes. For convenience, we first define a subset  $H_{LAP}(\Omega)\subset H^1_{loc}(\Omega)$ which contains all the functions that satisfy the radiation condition.


Following \cite{Fliss2015}, we first define the unbounded operator  in $L^2(\Omega)$ by
\[
 B=-\frac{1}{n+q}\Delta\text{ in }D(B):=\left\{\phi\in H^1(\Omega):\,\Delta u\in L^2(\Omega),\,\frac{\partial u}{\partial x_2}=0\text{ on }\Sigma_\pm\right\}.
\]
We need the following assumption to guarantee our theory.

\begin{assumption}
\label{asp0}
	For the positive valued  $k$, the equation $Bu=k^2 u$ only has a trivial solution in $D(B)$.
\end{assumption}

\begin{remark}
For perfectly periodic waveguide, an important result is that Assumption \ref{asp0} always holds (see \cite{Fliss2015,Kirsc2017a}).  {However, when there is a local perturbation, nontrivial solution may exist.} 


Now we show an example of  { a $k>0$ such that}  {$k^2$} lies in the  {point spectrum} of $B$ for  positive $n$ and $n+q$. Suppose $n$ and $f$ are defined by:
\begin{equation*}
	n(x)=\begin{cases}
		1,\quad |x-a_0|>0.3;\\
		9,\quad 0.1<|x-a_0|<0.3;\\
		1+8\,\zeta(|x-a_0|;0.1,0.3),\quad\text{otherwise.}
	\end{cases};\quad	f(x)=\begin{cases}
	0,\quad |x-a_0|>0.3;\\
	0.5,\quad 0.1<|x-a_0|<0.3;\\
	0.5\,\zeta(|x-a_0|;0.1,0.3),\quad\text{otherwise;}
\end{cases}
\end{equation*}
where $a_0=(0,0.5)^\top$, and $\zeta(t)$ is a  {$C^4$}-continuous function defined {by}
\begin{equation*}
	\zeta(t;a,b)=\begin{cases}
		1,\quad t\leq a;\\
		0,\quad t\geq b;\\
		1-\left[\int_{\tau=a}^b (\tau-a)^4(\tau-b)^4\d \tau\right]^{-1}\left[\int_{\tau=a}^t (\tau-a)^4(\tau-b)^4\d \tau\right],\, a<t<b.
	\end{cases}
\end{equation*}
When $k^2=3.2$, the problem \eqref{eq:waveguide} with $q=0$ is uniquely solvable in $H^1(\Omega)$. The solution $u$ decays exponentially when $|x_1|\rightarrow\infty$ thus the radiation condition defined in Definition \ref{def:rc} is satisfied. The  solution in $\Omega_0$ is shown in Figure \ref{fig:refractive_index_perturbed}, (a), which is strictly positive in $\overline{\Omega_0}$.  Let $q=-k^{-2}f/u$ (see (b) in Figure  \ref{fig:refractive_index_perturbed}),  {be the} local perturbation. As $n+q$ is strictly positive (see (c) in Figure \ref{fig:refractive_index_perturbed}),  {$0\neq u$} is the solution satisfying \eqref{eq:waveguide} with $f=0$ and  the radiation condition.

\end{remark}

\begin{figure}[ht]
	\centering
	\begin{tabular}{c  c c}
		\includegraphics[width=0.27\textwidth]{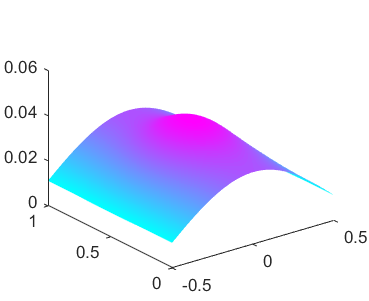} &
		\includegraphics[width=0.27\textwidth]{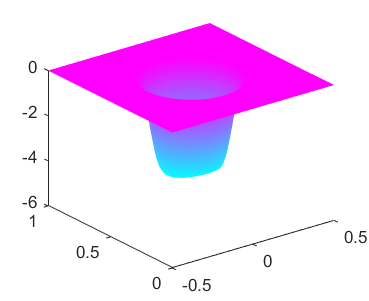} 
		& \includegraphics[width=0.27\textwidth]{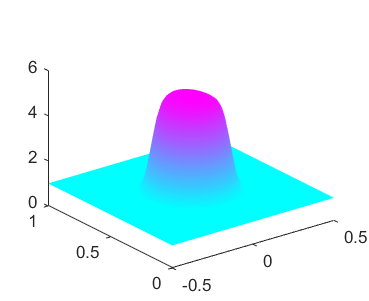}\\[-0cm]
		(a)&(b)&(c)
	\end{tabular}
	\caption{(a): numerical solution in $\Omega_0$; (b): the constructed local perturbation $q$; (c): the function $n+q$.}
	\label{fig:refractive_index_perturbed}
\end{figure}

\section{Explicit formulations of LAP solutions}

In this section, we formulate the LAP solutions  {using} two different methods introduced in \cite{Zhang2019b} and \cite{Kirsc2017a}, respectively. Note that the equation \eqref{eq:waveguide} can be rewritten as:
\begin{equation}\label{eq:waveguide_g}
\Delta u+k^2 n u =r\text{ in }\Omega,\quad \frac{\partial u}{\partial x_2}=0\text{ on }\partial\Sigma_\pm
\end{equation}
where $r=f-k^2q u$, which depends on $u$, is compactly supported in $\Omega_0$. In this section, we extend the formulations of LAP solutions with purely periodic refractive indexes to periodic ones with local perturbations, and also prove the unique solvability of the formulations.

\subsection{The  {$\alpha$-dependent periodic} problems}


 From the Floquet theory, the $\alpha$-dependent periodic problems are particularly interesting as they are  associated to the propagating modes (eigenfunctions).  
The strong formulation for the $\alpha$-dependent periodic problem is to find a periodic solution $v$ such that:
\begin{equation}\label{eq:per}
	\Delta v+2\i\alpha\frac{\partial v}{\partial x_1}+(k^2 n-\alpha^2) v=g(x)\,\text{ in }\Omega_0;\quad
	 \frac{\partial v}{\partial x_2}=0\text{ on }\Sigma_\pm^0,
\end{equation}
  {where $g\in L^2(\Omega_0)$. Note that  $g=\J r=e^{-\i\alpha x_1}r$ (for the definition of $\J$ we refer to the end of this subsection), since $r$ is compactly supported in $\Omega_0$.}
 For each $\alpha$,  the weak formulation of the periodic problems  {is to} find $v\in H^1_\p(\Omega_0)$ such that 
\begin{equation}\label{eq:per_var}
\int_{\Omega_0}\left[\nabla v\cdot\nabla\overline{\psi}-\i\alpha\left(\frac{\partial v}{\partial x_1}\overline{\psi}-v\frac{\partial\overline{\psi}}{\partial x_1}\right)-(k^2 n -\alpha^2)v\overline{\psi}\right]\d x=-\int_{\Omega_0}g\overline{\psi}\d x
\end{equation}
for any $\psi\in H^1_\p(\Omega_0)$. From Riesz representation theorem, there is an operator $A(\alpha,k):\,H^1_\p(\Omega_0)\rightarrow H^1_\p(\Omega_0)$ such that
\[
\left<A(\alpha,k)\phi,\psi\right>=\int_{\Omega_0}\left[\nabla \phi\cdot\nabla\overline{\psi}-\i\alpha\left(\frac{\partial \phi}{\partial x_1}\overline{\psi}-\phi\frac{\partial\overline{\psi}}{\partial x_1}\right)-(k^2 n -\alpha^2) \phi\overline{\psi}\right]\d x,
\]
where $\left<\cdot,\cdot\right>$ is the inner product in the space $H^1_\p(\Omega_0)$. It is obvious that $A(\alpha,k)$ is  a Fredholm operator (see \cite{Kirsc2017a}). When $k>0$ and $\alpha\in\R$, $A(\alpha,k)$ is a self-adjoint operator. There are operators $A_1,A_2,A_3,A_4$ which do not depend on $\alpha$ and $k$ with obvious definitions such that
\begin{equation}
 \label{eq:decomp_A}
 A(\alpha,k)=A_1+\alpha A_2+\alpha^2 A_3+k^2 A_4.
\end{equation}
Thus $A(\alpha,k)$ depends analytically on both $\alpha$ and $k$. From \cite{Kirsc2017a}, for fixed $k^2$, there are only finitely many  $\alpha$'s in $[-\pi,\pi]$ such that $A(\alpha,k)$ is not invertible, which are called exceptional values. The set of exceptional values is denoted by $S(k)$. They are solutions of the quadratic eigenvalue problems for fixed $k$, thus can be computed in a standard way (for details we refer to Section 4.2). In the following, we introduce some definitions and notations concerning exceptional values without proofs. For details we refer to \cite{Ehrhardt2009,Ehrhardt2009a,Fliss2015}.

For any $\alpha\in[-\pi,\pi]$, there is a family of analytic functions $\mu_i(\alpha)$ defined in $[-\pi,\pi]$ such that:
\[
 (A_1+\alpha A_2+\alpha^2 A_3)\phi=-\mu_i(\alpha) A_4.
\]
For 2D periodic waveguide, none of the analytic functions is constant (see \cite{Fliss2015}). For any index $i$, the graph of the function $\mu_i$, i.e., $\{\big(\alpha,\mu_i(\alpha)\big):\,\alpha\in[-\pi,\pi]\}$, is a dispersion curve. All the dispersion curves compose dispersion diagrams (see Figure \ref{fig:dd}).

\begin{figure}[ht]
	\centering
	\begin{tabular}{c  c}
		\includegraphics[width=0.35\textwidth]{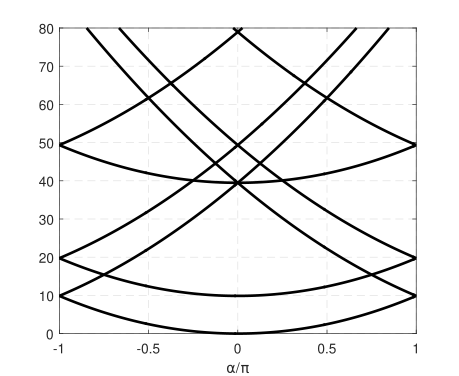} 
		& \includegraphics[width=0.35\textwidth]{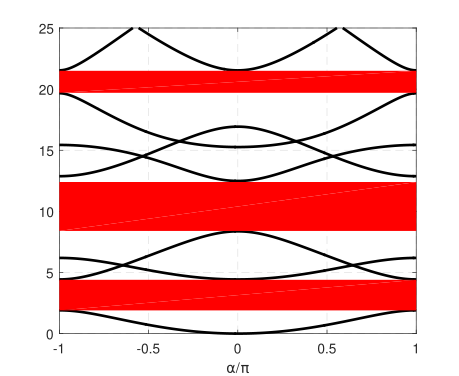}\\[-0cm]
	\end{tabular}
	\caption{Dispersion diagrams for different $n$}
	\label{fig:dd}
\end{figure}

For any fixed  $k>0$, there is a finite set $I$ ( {maybe} empty, for example when $k^2$ lies in the red bands on the right picture of Figure \ref{fig:dd}) such that  $S(k)=\Big\{\alpha\in[-\pi,\pi]:\,\exists\,i\in I,\,s.t., \,\mu_i(\alpha)=k^2\Big\}$.  
Corresponding to each dispersion curve $\mu_i(\alpha)$, there is also a family of eigenfunctions $\{\phi_i(\alpha,x):\,i\in I\}$ which also depend analytically on $\alpha$. When $\mu'_i(\alpha)>0$ ($\mu'_i(\alpha)<0$), then the corresponding eigenfunction $\phi_i(\alpha,\cdot)$ propagates to the right (left); when $\mu'_i(\alpha)=0$, then  $\phi_i(\alpha,\cdot)$ is a standing wave. 
For fixed $k>0$,  $S(k)$ is divided into the following subsets:
\begin{eqnarray*}
	&& S_-(k):=\left\{\alpha\in[0,2\pi]:\,\exists \, i\in I,\,s.t.,\,\mu_i(\alpha)=k^2,\,\mu'_i(\alpha)<0\right\};\\
	&& S_+(k):=\left\{\alpha\in[0,2\pi]:\,\exists \, i\in I,\,s.t.,\,\mu_i(\alpha)=k^2,\,\mu'_i(\alpha)>0\right\};\\	
	&& S_0(k):=\left\{\alpha\in[0,2\pi]:\,\exists \, i\in I,\,s.t.,\,\mu_i(\alpha)=k^2,\,\mu'_i(\alpha)=0\right\}.
\end{eqnarray*}
Note that when $S_0(k)\neq\emptyset$, the LAP does not work, so we have to make the following assumption.

\begin{assumption}\label{asp1}
In this paper, we assume that $S_0(k)=\emptyset$.
\end{assumption}

In \cite{Fliss2015}, it has been proved that the positive valued $k$'s such that Assumptions  \ref{asp1} holds is only a discrete set, thus this assumption is reasonable.

At the end of this subsection, we apply the Floquet-Bloch transform to \eqref{eq:waveguide}  (for details we refer to the appendix). Let us denote by $\J u$ the Floquet-Bloch transform of $u$ in $x_1$ direction (see Appendix), and set  $v(\alpha,x):=(\J u)(\alpha,x)$ where $\alpha\in[-\pi,\pi]$. With formal calculation, $v(\alpha,\cdot)$ satisfies \eqref{eq:per} and 
the solution $u$ is  given by the inverse transform of $v$:
\begin{equation}\label{eq:inverse}
 u(x)=\int_{-\pi}^{\pi}e^{\i\alpha x_1} v(\alpha,x)\d\alpha,\quad x\in\Omega_0.
\end{equation}

From above arguments, \eqref{eq:per_var} is uniquely solvable in $H^1_\p(\Omega_0)$ for all $\alpha\in [-\pi,\pi]\setminus S(k)$. From analytic Fredholm theory, the function $v$ is also extended to $\alpha\in\C$ (see Theorem 4 in \cite{Zhang2019b}).

\begin{theorem}
\label{th:per_depend}
 The Floquet-Bloch transformed field $v(\alpha,\cdot)=(\J u)(\alpha,\cdot)$ is extended to an analytic function for $\alpha\in \C\setminus\mathbb{F}$ (where $\mathbb{F}$ is a discrete set) and  {a} meromorphic function for $\alpha\in \C$. Moreover, $e^{\i\alpha x_1}v(\alpha,\cdot)$ is periodic with respect to the real part of $\alpha$, i.e.,
 \[
  e^{\i(\alpha+2\pi) x_1}v(\alpha+2\pi,\cdot)=e^{\i\alpha x_1}v(\alpha,\cdot),\quad \text{ for almost all }\alpha\in\C.
 \]
\end{theorem}

Note that when $S(k)\neq\emptyset$,  $v(\alpha,\cdot)$ does not  {exist} when $\alpha\in S(k)$ thus the inverse transform \eqref{eq:inverse} is not well defined. The formulation \eqref{eq:inverse} no longer works thus we need some modifications to give exact formulations for LAP solutions.


\subsection{Complex contour integral method}

In this subsection, we introduce the first method -- the complex contour integral  (CCI) method. To guarantee that the CCI method works, we have to make the following assumption.

\begin{assumption}\label{asp2}
	Assume that $k>0$ such that $S_-(k)\cap S_+(k)=\emptyset$.
\end{assumption}

From \cite{Zhang2019b}, the positive valued $k$'s such that Assumptions   \ref{asp2} is satisfied is only a discrete set.  

\begin{remark}
	Assumption \ref{asp2} is actually not necessary for the LAP. For purely periodic case,   the CCI method has been extended to the case without it in \cite{Zhang2019b}. Concerning the length of this paper, we keep this assumption to only focus on the most important part.
\end{remark}

Suppose that Assumption \ref{asp1} and \ref{asp2} are satisfied. 
The main idea of the CCI method is to modify the integral contour $[-\pi,\pi]$ in \eqref{eq:inverse} such that the points in $S(k)$ are avoided. 
As the set $S(k)$ is finite, from the symmetry between $S_+(k)$ and $S_-(k)$ (see \cite{Zhang2019b}), let
\[
S_+(k)=\Big\{\widehat{\alpha}_1^+,\dots, \widehat{\alpha}_N^+\Big\}\text{ and }S_-(k)=\Big\{\widehat{\alpha}_1^-,\dots, \widehat{\alpha}_N^-\Big\}\quad\text{ where } {\widehat{\alpha}_j^-}=-\widehat{\alpha}_j^+.
\]
First let the disk with center $\alpha$ and radius $\delta>0$ be denoted by $B(\alpha,\delta)$, and define
\[
 D_\delta:=\big[(-\pi,\pi)\times(0,+\infty)\big]\cup\left[\cup_{j=1}^N B\left(\widehat{\alpha}_j^+,\delta\right)\right]\setminus\left[\cup_{j=1}^N \overline{B\left(\widehat{\alpha}_j^-,\delta\right)}\right].
\]
Define the new integral contour by:
\[
\Lambda=\partial D_\delta\setminus\big[\{-\pi,\pi\}\times\R\big].
\]
With Assumption \ref{asp1} and \ref{asp2}, we choose $\delta>0$ such that the following conditions are satisfied:
\begin{itemize}
\item any two disks do not have nonempty intersection;
\item the closure of each disk  only contains one exceptional value $\widehat{\alpha}_j^\pm$.
\end{itemize}

Suppose the problem \eqref{eq:waveguide} has an LAP solution  $u\in H_{LAP}(\Omega)$, then $u$ is also the unique LAP solution of \eqref{eq:waveguide_g} with  $r:=f-k^2 q u$. In \cite{Zhang2019b}, the form of $u$ is given explicitly.   {Note that although $r$ depends on $u$, when $u$ is already a known LAP solution, we can still treat $r$ as some fixed function that is compactly supported in $\Omega_0$.}

\begin{theorem}[Theorem 7, \cite{Zhang2019b}]
Assumption \ref{asp1} and \ref{asp2} hold. Then the LAP solution for \eqref{eq:waveguide_g} is given by
\begin{equation}\label{eq:MD1_integral}
 u(x):=\int_\Lambda e^{\i\alpha x_1} v(\alpha,x)\d\alpha,\quad x\in\Omega.
\end{equation}
where $v(\alpha,\cdot)\in H^1_\p(\Omega_0)$ solves \eqref{eq:per_var}  for any fixed $\alpha\in\Lambda$ with  $g(x)=e^{-\i\alpha x_1}r(x) {=(\J r)(\alpha,x)}$. 
\end{theorem}

 {Now we have to prove that with Assumption \ref{asp0}, \ref{asp1} and \ref{asp2}, the problem \eqref{eq:MD1_integral} where $v(\alpha,\cdot)\in H^1_\p(\Omega_0)$ solves \eqref{eq:per_var}  for any fixed $\alpha\in\Lambda$ with  $g(x)=e^{-\i\alpha x_1}\left(f(x)-k^2 q (x)u(x)\right)$ has a unique solution in $H_{LAP}(\Omega)$, given any compactly supported $f\in L^2(\Omega_0)$. Then the unique solution is the LAP solution for \eqref{eq:waveguide}.

To describe the problem we first define the function space of the solution. Let $X:=L^2\left(\Lambda;H^1_\p(\Omega_0)\right)$. From the definition of $\Lambda$, the curve can be parameterized as $\Lambda:=\{s(t):=s_1(t)+\i s_2(t):\,t\in[0,1]\}$ where both $s_1$ and $s_2$ are real valued functions (for details we refer to Section 4.1). Then the norm of $X$ is defined as follows
\[
\|\phi\|_X=\left[\int_0^1\int_{\Omega_0} {\left[|\phi(s(t),x)|^2+|\nabla_x\phi(s(t),x)|^2\right]}|s'(t)|\d x\d t\right]^{1/2}.
\]

Since $v(\alpha,\cdot)$ satisfies\eqref{eq:per_var}, assume $\psi=\psi(\alpha,x)$ and integrating the equation on both sides with respect to $\alpha$ on the contour $\Lambda$, we arrive at the variational form of the problem.

 The weak formulation is to find $v\in X$ such that
\begin{equation*}
  \int_\Lambda\left<A(\alpha,k)v(\alpha,\cdot),\psi(\alpha,\cdot)\right>\d\alpha=-\int_{\Lambda}\int_{\Omega_0}e^{-\i\alpha x_1}r(x)\overline{\psi(\alpha,x)}\d x\d\alpha
\end{equation*}
for all $\psi\in X$. 
From Riesz representation theorem, there is an  { operator} $\A$  defined in $X$  such that
\[
 \left<\A \phi,\psi\right>_X=\int_\Lambda\left<A(\alpha,k)\phi(\alpha,\cdot),\psi(\alpha,\cdot)\right>\d\alpha,\quad\forall\,w,\,\psi\in X.
\]
 As any $\alpha\in\Lambda$ is not an exceptional value, $A(\alpha,k)$ is always invertible. Thus $\A$ is invertible in $X$.}

As $r=f-k^2 q u$, the variational formulation is written in the form: 
\begin{equation}
 \label{eq:var_MD1}
  \left<\A v,\psi\right>_X-k^2\int_\Lambda\int_{\Omega_0}e^{-\i\alpha x_1}q(x)u(x)\overline{\psi(\alpha,x)}\d x\d\alpha=-\int_\Lambda\int_{\Omega_0}e^{-\i\alpha x_1}f(x)\overline{\psi(\alpha,x)}\d x\d\alpha.
\end{equation}
For simplicity, define the following operators. Let 
\[
(\K v)(x):=\int_\Lambda e^{\i\alpha x_1}v(\alpha,x)\d\alpha,\quad x\in\Omega_0.
\]
From Riesz representation theorem, there is an operator $\T$ in $H^1_\p(\Omega_0)$ such that 
\[ \left<\T\phi,\psi\right>=\int_{\Omega_0}q(x)\phi(x)\overline{\psi(x)}\d x,\,\text{ for any }\phi,\,\psi\in H^1_\p(\Omega_0).
\]
As  {$H^1(\Omega_0)$} is  {compactly} embedded in  {$L^2(\Omega_0)$}, the operator $\T$ is also compact. Similarly, define $\L$  by:
\[
\left<\L\phi,\psi\right>=\int_{\Omega_0}\phi(x)\overline{\psi(x)}\d x,\,\text{ for any }\phi\in L^2(\Omega_0),\,\psi\in H^1_\p(\Omega_0).
\]
Then $\L$ is bounded from $L^2(\Omega_0)$ to $H^1_\p(\Omega_0)$.
Thus \eqref{eq:var_MD1} is equivalent to
\begin{equation}
    \label{eq:MD1}
    (\A-k^2\K^*\T\K) {v}=-\K^* \L f.
\end{equation}

Now we study the property of the operator $\K$. We begin with a classical Minkowski integral inequality. 

\begin{lemma}[\cite{Hardy1988}, Theorem 202]
\label{lm:Minkowski} 
Suppose $(S_1, \mu_1 )$ and $(S_2 , \mu_2 )$ are two measure spaces and $F : S_1 \times S_2 \rightarrow \R$
is measurable. Then the following inequality holds for any $p \geq 1$
\begin{equation}
 \left[\int_{S_2}\left|\int_{S_1}F(y,z)\d\mu_1(y)\right|^p\d \mu_2(z)\right]^{1/p}
 \leq \int_{S_1}\left(\int_{S_2}|F(y,z)|^p\d \mu_2(z)\right)^{1/p}\d\mu_1(y).
\end{equation}

\end{lemma}

\begin{lemma}\label{lm:k}
 The operator $\K$ is bounded from $L^2(\Lambda;H^m(\Omega_0))$ to $H^m(\Omega_0)$ for any fixed non-negative { integer $m$}. Especially, it is bounded from $X$ to $H^1(\Omega_0)$.
\end{lemma}

\begin{proof} {
 First consider the case when $m=0$. Given any $w\in C^\infty(\Lambda\times\Omega_0)$, then $\K w$ is well defined and uniformly bounded for any $x\in\Omega_0$. Recall the parameterization of $\Lambda$,
 \[
  \|\K w\|_{L^2(\Omega_0)}=\left[\int_{\Omega_0}\left|\int_{\Lambda}e^{\i\alpha x_1}w(\alpha,x)\rm{d}\alpha\right|^2\d x\right]^{1/2}=\left[\int_{\Omega_0}\left|\int_{0}^1 e^{\i s(t) x_1}w(s(t),x)s'(t)\rm{d}t\right|^2\d x\right]^{1/2}.
 \] 
 Then {from Lemma \ref{lm:Minkowski}} (with $p=1$) and the Cauchy-Schwartz inequality,
 \begin{align*}
  \|\K w\|_{L^2(\Omega_0)}&=\left[\int_{\Omega_0}\left|\int_{0}^1 e^{\i s(t) x_1}w(s(t),x)s'(t)\rm{d}t\right|^2\d x\right]^{1/2}
  \leq \int_0^1\left(\int_{\Omega_0}|e^{\i s(t) x_1}w(s(t),x)|^2|s'(t)|^2\d x\right)^{1/2}\d t\\
  &\leq \left(\int_0^1\int_{\Omega_0}|e^{\i s(t) x_1}w(s(t),x)|^2|s'(t)|\d x \d t\right)^{1/2}\left( {\int_0^1} |s'(t)| \d t\right)^{1/2}
  \leq C\|w\|_{L^2(\Lambda;L^2(\Omega_0))}.
 \end{align*}
From the density of $C^\infty(\Lambda\times\Omega_0)$ in $L^2(\Lambda;L^2(\Omega_0))$, the above inequality holds for all $w\in L^2(\Lambda;L^2(\Omega_0))$. Thus $\K$ is bounded from $L^2(\Lambda;L^2(\Omega_j))$ to $L^2(\Omega_j)$. 
For $m\geq 1$, the proof is similar thus is omitted. So $\K$ is bounded from $L^2(\Lambda;H^m(\Omega_0))$ to $H^m(\Omega_0)$ and particularly it is bounded from $X$ to $H^1(\Omega_0)$.
}
\end{proof}

As $\A$ is invertible, $\K$ and $\T$ are bounded and $\T$ is compact, $\A-k^2\K^*\T\K$ is a Fredholm operator. Thus it is invertible if and only if it is an injection. Then we obtain the  well-posedness of the problem \eqref{eq:MD1} in the following theorem.

\begin{theorem}\label{th:inv_cci}
 With Assumption \ref{asp0}, the operator $\A-k^2\K^*\T\K$ is invertible. That is, given any compactly supported $f\in L^2(\Omega_0)$, the problem \eqref{eq:MD1} has a unique solution in $X$.
\end{theorem}

\begin{proof}
 Suppose $v$ is a solution of \eqref{eq:MD1} with $f=0$, i.e.,  $(\A-k^2\K^*\T\K)v=0$. Then $u=\K v$ lies in $H^1(\Omega_0)$. From \cite{Zhang2019b}, $u=\K v$ is the LAP solution  {$r=-k^2 q u$ in \eqref{eq:waveguide_g}} thus it is  the solution of \eqref{eq:waveguide} with $f=0$ and satisfies the radiation condition  {of} Definition \ref{def:rc0}. From Assumption \ref{asp0}, $u=0$. Thus $v=0$, which implies that $\A-k^2\K^*\T\K$ is injective thus is invertible. So the problem \eqref{eq:MD1} is well-posed in the space $X$.
\end{proof}

We have further regularity results for the solution $ {v}$.

\begin{corollary}
 \label{cr:cci}
 With Assumption \ref{asp0}, given any compactly supported $f\in L^2(\Omega_0)$. The solution $ {v}$ depends piecewise smoothly on $\alpha\in\Lambda$ and for fixed $\alpha$, $ {v}(\alpha,\cdot)\in H^2_\p(\Omega_0)$.
\end{corollary}

\begin{proof}
 As for any $\alpha\in\Lambda$, $A(\alpha,k)$ is invertible and $\Lambda$ is a piecewise smooth curve, $A(\alpha,k)$ depends piecewise smoothly on $\alpha$ from the perturbation theory. For each fixed $\alpha\in\Lambda$, $ {v}(\alpha,\cdot)\in H^2_\p(\Omega_0)$ from interior regularity.
 
\end{proof}

\subsection{Decomposition method}

In \cite{Kirsc2019a,Kirsc2019b} { (also see [34], Chapter 5, paragraph 2)}, an explicit formulation of the radiation condition is given for periodic open tubes in 3D. However, since the method is easily transferred to this case, we  introduce the decomposition method based on the  definition without proofs. 
For simplicity, let
\[
S(k):=\left\{\widehat{\beta}_j:\,j\in J\right\}\text{ where $J$ is a finite set}.
\]
 {Although  this set has already been introduced in previous subsections, we use different notations to indicate the elements in this set, just to avoid  confusions.} For any fixed $j$,   the Fredholm operator $A(\widehat{\beta}_j,k)$ is not an injection,  and the null space $\widehat{Y}_j:=\mathcal{N}(A(\widehat{\beta}_j,k))$ is finite dimensional with dimension $m_j$.  
There is an orthonormal eigensystem in $\widehat{Y}_j$
 \begin{equation}\label{eq:orth_sys}
   \Big\{(\lambda_{\ell,j},\widehat{\phi}_{\ell,j}):\,\ell=1,2,\dots,m_j\Big\}
 \end{equation}
such that
\begin{equation}
 \label{eq:orth_relation}
 \int_{\Omega_0}\left[-\i\frac{\partial \widehat{\phi}_{\ell,j}}{\partial x_1}+\widehat{\beta}_j\widehat{\phi}_{\ell,j}\right]\overline{\psi}\d x=\lambda_{\ell,j} k\int_{\Omega_0}n\widehat{\phi}_{\ell,j}\overline{\psi}\d x\text{ for all }\psi\in \widehat{Y}_j 
\end{equation}
with normalization
\begin{equation}
	\label{eq:normal}
 2k\int_{\Omega_0}n\widehat{\phi}_{\ell,j}\overline{\widehat{\phi}_{\ell',j}}\d x=\delta_{\ell,\ell'}\text{ for }\ell,\ell'=1,2,\dots,m_j.
\end{equation}
 {Note that from the dispersion diagram, the positive integer $m_j$ indicates the number of dispersion curves that pass through the point $\left(\widehat{\beta}_j,k^2\right)$.
We can use the values of $\lambda_{\ell,j}$ to the direction of the eigenfunctions. When $\lambda_{\ell,j}>0$ ($<0$), the eigenfunction $\widehat{\phi}_{\ell,j}$ propagates to the right (left) and $\widehat{\beta}_j\in S_+(k)$ ($S_-(k)$); when $\lambda_{\ell,j}=0$, $\widehat{\phi}_{\ell,j}$ is a standing wave thus $\widehat{\beta}_j\in S_0(k)\neq\emptyset$. Since Assumption \ref{asp1} holds, $\lambda_{\ell,j}\neq 0$ for all possible $\ell$ and $j$. In this case, it is possible that for some $j\in J$, $\lambda_{j,\ell}>0$ and $\lambda_{j,\ell'}<0$ with $\ell\neq\ell'$, which means Assumption \ref{asp2} is not necessary for the decomposition method.
}

Let $\phi_{\ell,j}:=e^{\i\widehat{\beta}_j x_1} {\widehat{\phi}_{\ell,j}}$,
then $\phi_{\ell,j}$ is a $\widehat{\beta}_j$-quasi-periodic eigenfunction to the Helmholtz equation
\[
 \Delta \phi_{\ell,j}+k^2 n \phi_{\ell,j}=0\text{ in }\Omega_0;\quad \phi_{\ell,j}=0\text{ on }\Sigma_\pm^0.
\]
For each $j\in J$, let the space spanned by $\{\phi_{\ell,j},\,\ell=1,2,\dots,m_j\}$ be denoted by $Y_j$.

With above notations, we  {are prepared to} recall the radiation condition which is equivalent, but more explicit, compared to that defined in Definition \ref{def:rc0}.

\begin{definition}[Theorem 3.7, \cite{Kirsc2019a}]\label{def:rc}
Suppose Assumption \ref{asp0} and \ref{asp1} hold.The LAP solution $u$  has a decomposition $u=u^{(1)}+u^{(2)}$ where $u^{(1)}\in H^1(\Omega)$ and $u^{(2)}$ has the form
\begin{equation}\label{eq:u2_decomp}
u^{(2)}(x)=\psi^+(x_1)\sum_{j\in  {J}}\sum_{\lambda_{\ell,j}>0}f^+_{\ell,j}\phi_{\ell,j}(x)+\psi^-(x_1)\sum_{j\in  {J}}\sum_{\lambda_{\ell,j}<0}f^-_{\ell,j}\phi_{\ell,j}(x)
\end{equation}
for some $f_{\ell,j}^\pm\in \C$ defined by (see \cite{Kirsc2019b})
\begin{equation}\label{eq:MD2_coefficient}
	f_{\ell,j}^\pm=	f_{\ell,j}\ =\ \frac{\i}{|\lambda_{\ell,j}|}\int_{\Omega_0}
r\,\overline{\phi_{\ell,j}}\,\d x.
\end{equation}
 and 
\begin{equation*}
\psi^+(x_1)\rightarrow 1\text{ as }x_1\rightarrow\infty,\,
\psi^+(x_1)\rightarrow 0\text{ as }x_1\rightarrow-\infty;\quad \psi^-(x_1)=\psi^+(-x_1).
\end{equation*}
\end{definition}


 {
	To simplify the  {representation}, we assume that ${\rm supp}(q)\subset(-1/2+\delta,1/2-\delta)\times(0,1)$ for a $\delta\in(0,1/2)$, we can choose proper $\psi^+$ and $\psi^-$ such that $\psi^\pm(x_1)q(x)=0$ for all $x\in\Omega$. For example, we can choose 
	\[
	\psi^+(x_1)=1\text{ for }x_1\geq 1-\frac{\delta}{4}\text{ and }\psi^+(x_1)=0\text{ for }x_1\leq 1-\frac{3\delta}{4}.
	\]}
 {Replacing} $u^{(2)}$ by its decomposition \ref{eq:u2_decomp} in \eqref{eq:waveguide_g}, we  easily arrive at the following equation for $u^{(1)}$:
\begin{equation}\label{eq:u1_1}
\Delta u^{(1)}+k^2 n u^{(1)}  = \M(r),
\end{equation}
where 
\[
 \M(r):=r-(\Delta+k^2n)u^{(2)}=r-\sum_{j\in J}\sum_{j=1}^m f_{\ell,j}\, g_{\ell,j}(x)
\]
with the functions $g_{\ell,j}$  defined by:
\begin{equation*}
g_{\ell,j}(x)=(\Delta+k^2n(x))\left[\psi^\pm(x_1)\phi_{\ell,j}(x)\right]=
  2(\psi^\pm(x_1))'\frac{\partial\phi_{\ell,j}}
{\partial x_1}(x)+(\psi^\pm(x_1))^{''}\,\phi_{\ell,j}(x)\text{ when }\pm\lambda_{\ell,j}>0.        
\end{equation*}
From the property of $\psi^\pm$, ${\rm supp}\left(g_{\ell,j}\right)\subset \Omega_0$ for all $j$ and $\ell$, thus $\M(r)$ is also compactly supported in $\Omega_0$.  
 {
It is easily checked that $\M$ is a bounded linear operator in $L^2\left(\Omega_0\right)$ and ${\rm ran}(\M)$ is orthogonal to $Y_j$ for any $j\in J$. For details we refer to Theorem 4.4 in \cite{Kirsc2019b}.} 

Let $v(\beta,x):=\left(\J u^{(1)}\right)(\beta,x)$. Since $u^{(1)}$ decays exponentially as $|x_1|\rightarrow\infty$, $v(\beta,\cdot)\in H^1_\p(\Omega_0)$ exists for all $\beta\in[-\pi,\pi]$ and depends analytically on $\beta$ in the interval (and also  {extend} analytically to a small neighbourhood of $[-\pi,\pi]$). For any fixed $\beta$, it is the unique solution of \eqref{eq:per_var} with $r$  { replaced} by $\M (r)$.  

 {Similarly to} the CCI method, the weak formulation for the problem is to find $v\in L^2((-\pi,\pi);H^1_\p(\Omega_0))$ such that
\begin{equation}
 \label{eq:MD2_var}
 \begin{aligned}
  \int_{-\pi}^\pi\left<A(\beta,k)v(\beta,\cdot),{\psi(\beta,\cdot)}\right>\d\beta-k^2\int_{\Omega_0}\M\left(q u^{(1)}\right)\overline{\left(\int_{-\pi}^\pi e^{\i\beta x_1}\psi(\beta,x)\d\beta\right)}\d x\\=-\int\limits_{\Omega_0}\M(f)\,\overline{\left(\int_{-\pi}^\pi e^{\i\beta x_1} \psi(\beta,x)\d\beta\right)}\,\d x.
 \end{aligned}
\end{equation}
 {Define the operator $\B$ by:
\[
\left<\B v,\phi\right>=\int_{-\pi}^\pi\left<A(\beta,k)v(\beta,x),{\phi(\beta,x)}\right>\d\beta,\quad\text{ for all }v,\,\phi\in L^2((-\pi,\pi);H^1_\p(\Omega_0)).
\]
Then with the definition of $\T$ and $\L$ in the previous section, 
the variational form \eqref{eq:MD2_var} is now equivalent to
\begin{equation}\label{eq:MD2}
	\left(\B-k^2 \J\M\T\J^{-1}\right)v=-\J\M\L f.
\end{equation}

As $A(\beta,k)$ is self-adjoint for real valued $\beta$, $\B$ is also self-adjoint. Thus the space $L^2((-\pi,\pi);H^1_\p(\Omega_0))$ has the following decomposition:
\[
L^2((-\pi,\pi);H^1_\p(\Omega_0))={\rm ran}(\B)\oplus\ker(\B)\text{ and }{\rm ran}(\B)\bot\ker(\B),
\]
and $\B$ is an isomorphism from $Y:={\rm ran}(\B)$ to itself.
Moreover, ${\rm ker}(\B)=\bigoplus_{j\in J}\widehat{Y}_j$.  Since ${\rm ran}(\M)$ is orthogonal to all $Y_j$, ${\rm ran}(\J\M)$ is orthogonal to $\widehat{Y}_j$, which implies that ${\rm ran}(\J\M)\subset Y$. Thus the operator $\B-k^2 \J\M\T\J^{-1}$ is  {an endomorphism of} $Y$.
As  $\T$ is compact, the operator $\J\M\T\J^{-1}$ is a compact operator with range in $Y$. As $\B$ is invertible in $Y$, $\B-k^2 \J\M\T\J^{-1}$ is a Fredholm operator. Thus we arrive at the  well-posed result of \eqref{eq:MD2}.

\begin{theorem}
 \label{th:inver_rc}
  {Under assumption \ref{asp0} and \ref{asp1}}, then the operator $\B-k^2 \J\M\T\J^{-1}$ is invertible in $Y$.  Given any compactly supported $f\in L^2(\Omega_0)$, the problem \eqref{eq:MD2} has a unique solution in $Y$.
\end{theorem}}

When  {$v$} is the unique solution of \eqref{eq:MD2},  {then $u$} satisfies the radiation condition introduced in Definition \ref{def:rc}. Thus it is the LAP solution for the problem \eqref{eq:waveguide}.

 {Finally}, we have to introduce a special technique  {for} solving \eqref{eq:MD2} numerically.  {Since $A(\beta,k)$ is not invertible in $H^1_\p(\Omega_0)$ when $\beta=\widehat{\beta}_j$, it is difficult to deal with  {such a point}.} To this end, we  modify the formulation  to avoid the singularities.  Note that $v(\beta,\cdot)$ depends analytically on $\beta$ in $(-\pi-\delta,\pi+\delta)\times(-2\sigma,2\sigma)$ for sufficiently small $\delta,\,\sigma>0$. From Cauchy integral formula,
\[
 \oint_{C_\sigma} e^{\i\beta x_1}v(\beta,x)\d\beta=0,
\]
where $C_\sigma$ is the boundary of the rectangle $[-\pi,\pi]\times[0,\sigma]$ which is oriented counter-clockwise. 
On the other hand, $e^{\i\beta x_1}v(\beta,\cdot)$ is $2\pi$-periodic with respect to the real part of $\beta$, this implies that the integrals on the left- and right edges cancel. Thus 
\[
 \int_{-\pi}^\pi e^{\i\beta x_1}v(\beta,x)\d\beta= \int_{-\pi}^\pi e^{\i(\beta+\i\sigma) x_1}v(\beta+\i\sigma,x)\d\beta
\]
where $v(\beta+\i\sigma,\cdot)$ solves \eqref{eq:per_var} with parameter $\beta+\i\sigma$. In the numerical schemes, to avoid the singularities we always replace $\beta$ by $\beta+\i\sigma$ in the variational formulation \eqref{eq:MD2_var}.

\section{Numerical schemes}

In this section, we introduce two numerical schemes based on the two representations. For both methods,  the periodic problems \eqref{eq:per_var} are computed several times. So we briefly recall the finite element method for these problems at the very beginning. 

Let $\M_h$ be a family of regular curved and quasi-uniform meshes which cover the domain $\Omega_0$. We also assume  {that the} number and heights of nodal points 
on the left and right boundaries of $\Omega_0$ are the same,  {hence} we can define functions that can be extended to periodic ones on the meshes. By omitting nodal points on the right boundary, we assume that the points $x_j$ ($j=1,2,\dots,M'$) be all the nodal points, where $M'$ is a positive integer. Let $M>M'$ be a positive integer  and $x_j$ ($j=M'+1,\dots,M$)  {the} points on the right boundary. Let $\{{\zeta}_j:\,j=1,2,\dots,M'\}$ be the family of piecewise linear and globally continuous basis functions that vanish on $\Sigma_\pm^0$ defined on the meshes $\M_h$, and ${\zeta}_j(x_\ell)=\delta_{j,\ell}$ where $j,\ell=1,2,\dots,M'$ and where $\delta_{j,\ell}$ is the Kronecker delta function. Define the function $\zeta_j$ on $\Gamma_1$ by the value on $\Gamma_0$, then it is also extended periodically to $\Omega$ as a $1$-periodic function. Then we define the finite dimensional subspace: 
\[
V_h:={\rm span}\Big\{\zeta_j:\,j=1,2,\dots,M'\Big\}\subset H^1_\p(\Omega_0).
\]
Then the solutions   {$v(\alpha,\cdot)$} are approximated in the finite dimensional subspace $V_h$. 

In the next subsections, we introduce the two numerical methods based on formulations of the LAP solution  by the complex contour integral and decomposition method.  {To simplify the process, we require the further assumption.
	\begin{assumption}\label{asp3}
		Assume that for any element  $\alpha\in S(k)$, there is only one dispersion curve passing through $(\alpha,k^2)$.
	\end{assumption}

Note that although Assumption \ref{asp3} is not necessary for both methods, the set with all the positive wave numbers such that Assumption \ref{asp3} does not hold is only a zero set. 
}

\subsection{The CCI method}

First we recall the variational form for the CCI method, i.e.,  { find} $v\in X$ such that
\begin{eqnarray}
    \label{eq:MD1_1}
    &&\begin{aligned}
    \int_\Lambda\left<A(\alpha,k)v(\alpha,\cdot),\psi(\alpha,\cdot)\right>\d\alpha-k^2\int_\Lambda\int_{\Omega_0}e^{-\i\alpha x_1}q(x)u(x)\overline{\psi(\alpha,x)}\d x\d\alpha\\=-\int_\Lambda\int_{\Omega_0}e^{-\i\alpha x_1}f(x)\overline{\psi(\alpha,x)}\d x\d\alpha
    \end{aligned}
    \\\label{eq:MD1_2}
   && u(x)=\int_\Lambda e^{ {\i\alpha x_1}}v(\alpha,x)\d\alpha.
\end{eqnarray}

In this subsection, we  focus of the discretization  {of} the curve $\Lambda$. 
The modified integral contour $\Lambda$ is composed  {of a} finite number of intervals and semicircles, then the first step is to parameterize $\Lambda$ piecewisely:
\[
\Lambda=\cup_{j=1}^Q \overline{\Big\{s_j(t)=s_j^1(t)+\i s^2_j(t):\,t\in I_j\Big\}},
\]
where $I_j$ are closed bounded intervals. For each fixed $j=1,2\dots,Q$, $s_j^1$ and $s^2_j$  {are real smooth functions on the interval $I_j$}. With a proper parameterization,  let the intervals $I_1,\dots,I_Q$  be  
\[
I_1=(0,a_1),\,I_2=(a_1,a_2),\dots, I_j=(a_{j-1},a_j),\dots, I_Q=(a_{Q-1},a_Q)=(a_{Q-1},1),
\]
 then 
\[
\Lambda=\Big\{s(t)=s_1(t)+\i s_2(t):\,t\in [0,1]\Big\}
\]
where $s_1=s_j^1$ and $s_2=s_j^2$ in each $I_j$, respectively. Since the trapezoidal rule converges much faster for smooth periodic functions  than nonperiodic smooth functions (see \cite{Weide2002}), we require
\[
s^{(m)}_j\left(a_{j-1}\right)=s^{(m)}_j\left(a_{j}\right)=0,\quad\forall\,j=1,2,\dots,Q\text{ and }m=1,2.
\]
Then $s'(t)$ is smooth in $[0,1]$ and it can be extended to a periodic smooth function in $\R$. For details of this technique we refer to Section 9.6 in \cite{Kress1998}.

Replacing $\alpha$ by $s(t)$ in the equations  \eqref{eq:MD1_1}-\eqref{eq:MD1_2} and  {letting} $\widetilde{v}(t,x):=v(s(t),x) s'(t)$, we arrive at the following equivalent equations:
\begin{eqnarray}
    \label{eq:MD1_1t}
    &&\begin{aligned}
    \int_0^1\left<A(s(t),k)\widetilde{v}(t,\cdot),\psi(t,\cdot)\right> \d t-k^2\int_0^1\int_{\Omega_0}e^{-\i s(t) x_1}q(x)u(x)\overline{\psi(t,x)}s'(t)\d x\d t\\=-\int_0^1\int_{\Omega_0}e^{-\i s(t) x_1}f(x)\overline{\psi(t,x)}s'(t)\d x\d t,
    \end{aligned}
    \\\label{eq:MD1_2t}
   && u(x)=\int_0^1 e^{\i s(t) x_1}\widetilde{v}(t,x)\d t.
\end{eqnarray}
The variational problem is to { seek } $\widetilde{v}\in \widetilde{X}:=L^2\left((0,1);H^1_\p(\Omega_0)\right)$ such that \eqref{eq:MD1_1t}-\eqref{eq:MD1_2t} hold for all test function  {$\psi\in\widetilde{X}$. From} Corollary \ref{cr:cci}, $\widetilde{v}\in C^\infty_\p\left([0,1];H^2_\p(\Omega_0)\right)$.

Now we discretize the system \eqref{eq:MD1_1t}-\eqref{eq:MD1_2t}. As the finite element discretization in the domain $\Omega_0$ has already been introduced  at the beginning of this section, we only introduce the trigonometric interpolation in the domain $[0,1]$. 
Let $N$ be a positive integer and the uniformly distributed nodal points be
\[
t_0=0;\quad t_j=\frac{j}{N}\quad\forall\,j=1,2,\dots,N.
\]
Suppose $N$ is an even number,  {then let us introduce the basic functions:}
\[
\xi_\ell(t):=\frac{1}{N}\sum_{m=-N/2+1}^{N/2}\exp\left(\i2 \pi m(t-t_\ell)\right),\,\ell=1,2,\dots,N.
\]
It is well  {known} that 
\[
\xi_\ell(t_{\ell'})=\delta_{\ell,\ell'}\text{ and }\int_0^1 \xi_\ell(t)\xi_{\ell'}(t)\d t=\frac{1}{N}\delta_{\ell,\ell'}.
\]

Now we are prepared to discretize the system \eqref{eq:MD1_1t}-\eqref{eq:MD1_2t} in the finite dimensional space:
\[
\widetilde{X}_{N,h}:=\underbrace{ V_h \bigoplus V_h \bigoplus\cdots \bigoplus V_h }_{N\text{ spaces }}.
\]
Let $\widetilde{v}_{N,h}\in \widetilde{X}_{N,h}$ be the approximation of $\widetilde{v}$  {of the form:}
\[
\widetilde{v}_{N,h}(t,x)=\sum_{\ell=1}^N\sum_{j=1}^M \widehat{v}_{\ell,j}\xi_\ell(t)\zeta_j(x),
\]
 {then
\begin{equation}
	\label{eq:MD1_2_fem}
u_{N,h}(x)=\sum_{j=1}^M \widehat{u}_j\zeta_j (x)\text{ where }\widehat{u}_j=\frac{1}{N}\sum_{\ell=1}^{N}e^{\i s(t_\ell)x_1}\widehat{ {v}}_{\ell,j}.
\end{equation}
}
With the test function $\psi(t,x)=\xi_{\ell'}(t)\zeta_{j'}(x)$, \eqref{eq:MD1_1t}-\eqref{eq:MD1_2t} is discretized as:
\begin{equation}
\label{eq:MD1_1_fem}
\begin{aligned}
\frac{1}{N}\delta_{\ell,\ell'}\sum_{j=1}^{M'}\widehat{v}_{\ell,j}\left<A(t_\ell,k)\zeta_j,\zeta_{j'}\right>-\frac{k^2}{N}\sum_{j=1}^{{M'}}\widehat{u}_j \left(\int_{\Omega_0}e^{-\i s(t_{\ell'}) x_1}q(x)\zeta_j(x) {\zeta_{j'}(x)}\d x\right)s'(t_{\ell'})\\=-\frac{1}{N}\left(\int_{\Omega_0}e^{-\i s(t_{\ell'}) x_1}f(x)\zeta_{j'}(x)\d x\right)s'(t_{\ell'}).
\end{aligned}
\end{equation}

With \eqref{eq:MD1_2_fem} we also get the approximation of $u$, i.e., $u_{N,h}$, at the same time.

Finally the system \eqref{eq:MD1_2_fem}-\eqref{eq:MD1_1_fem} is summarized  {by} the following  {system}:
\begin{equation}
 \label{eq:MD1_matrix}
 \left(
 \begin{matrix}
  A_1 & 0 & \cdots &0 & C_1\\
  0 & A_2 & \cdots &0 & C_2\\
  \vdots & \vdots & \ddots & \vdots &\vdots\\
  0 & 0 & \cdots & A_{N} & C_{N}\\
  B_1 & B_2 & \cdots& B_{N} & I
 \end{matrix}
 \right)
  \left(
 \begin{matrix}
  V_1 \\ V_2 \\ \vdots \\ V_{N}\\ U
 \end{matrix}
 \right)= \left(
 \begin{matrix}
  F_1 \\ F_2 \\ \vdots \\ F_{N}\\ 0
 \end{matrix}
 \right),
\end{equation}
where $V_j=\left(\widehat{v}_{1,j},\dots,\widehat{v}_{M,j}\right)^\top$ and $U=\left(\widehat{u}_{1},\dots,\widehat{u}_{M}\right)^\top$. $A_j$ comes from the first term of \eqref{eq:MD1_1_fem}, $C_j$ comes from the second term of \eqref{eq:MD1_1_fem}, $B_j$ comes from \eqref{eq:MD1_2_fem}, $F_j$ comes from the right hand side. 
The size of this system is $(N+1)M\times (N+1)M$.
Note that $U$ is treated as an additional unknown vector just for a simpler representation of the linear system. Then the final task is to solve the linear sparse system with the size $(N+1)M\times(N+1)M$.

\subsection{The decomposition method}

 We introduce the decomposition method in this subsection. Since Assumption \ref{asp3} holds, $m_j=1$ for all $j\in J$, so we abbreviate the notations $\lambda_{\ell,j},\,\phi_{\ell,j},\,g_{\ell,j}$ as $\lambda_{j},\,\phi_j,\,g_j$.  First we define
	\begin{equation}\label{eq:f0}
		f_0(x):=\M(f)=f-\sum_{j\in J}\frac{\i}{|\lambda_{j}|}\left[\int_{\Omega_0}f(x)\overline{\phi_{j}}\d x\right] g_{j}(x).\end{equation}
We simplify the second term in \eqref{eq:MD2_var}:
\[
\int_{\Omega_0}\M\left(qu^{1}\right)\overline{\int_{-\pi}^\pi e^{\i\beta x_1}\psi(\beta,x)\d\beta}\d x=\int_{\Omega_0} u^{(1)}(x)\overline{\theta(x;\psi)}\d x
\]	
where $\theta(x;\psi)=q(x)\M^*\left[\int_{-\pi}^\pi e^{\i\beta x_1}\psi(\beta,x)\d\beta\right]$ and $\M^*$ is the adjoint operator of $\M$  defined as:
\[
\M^*(f)=f(x)+\sum_{j\in J}\frac{\i}{|\lambda_{j}|}\phi_j(x)\left[\int_{\Omega_0}\overline{g_j(y)}f(y)\d y.\right]
\]

	From the variational form  \eqref{eq:MD2_var}  (note that $\beta$ is already replaced by $\beta+\i\sigma$):
	\begin{eqnarray}
		\label{eq:MD2_1}
		&&\begin{aligned}
			\int_{-\pi}^\pi \left<A(\beta+\i\sigma,k)v(\beta {+\i\sigma},\cdot),\psi(\beta {+\i\sigma},\cdot)\right>\d\beta-k^2 \int_{\Omega_0} u^{(1)}(x)\overline{\theta(x;\psi)}\d x\\
			=-\int_{-\pi}^\pi\int_{\Omega_0}e^{-\i(\beta+\i\sigma) x_1} f_0(x)\overline{\psi(\beta {+\i\sigma},x)}\d x\d \beta;
		\end{aligned}
		\\\label{eq:MD2_2}
		&& u^{(1)}(x)=\int_{-\pi}^\pi e^{\i(\beta+\i\sigma) x_1}v(\beta+\i\sigma,x)\d\beta.
	\end{eqnarray}

	 {To discretize \eqref{eq:MD2_1}-\eqref{eq:MD2_2}, the first step is to  approximate the exceptional values $\widehat{\beta}_j$ and the corresponding systems
$\left\{\left(\lambda_{j},\widehat{\phi}_{j}\right)\right\}$. }
This implementation is carried out by the following two steps. The first step is to find all the real eigenvalues and corresponding eigenspaces of the quadratic pencil  $A(\alpha,k)=A_1+k^2 A_4+\alpha A_2+\alpha^2 A_3$ (see  \eqref{eq:decomp_A}).  The operators are discretized by the finite element method and let $A_1^h,A_2^h,A_3^h,A_4^h,B_1^h,B_2^h$ be the corresponding matrices. A standard way to solve above quadratic eigenvalue problem is to solve the following linearized problems:
	\begin{equation}\label{eq:gep}
		B_1^h W^h=\lambda B_2^h W^h,
	\end{equation}
	where \[
	B_1^h:=\left(\begin{matrix}
		A_1^h+k^2 A_4^h & 0\\ 0 & I     
	\end{matrix}
	\right),\quad B_2^h=\left(\begin{matrix}
		-A_2^h & -A_3^h \\ I & 0
	\end{matrix}\right).
	\]
By solving this problem, we find all the eigenvalues and eigenfunctions 
$$\left\{\left(\widehat{\beta}_j^h,\widehat{\phi}_{j}^h\right):\,j\in J^h\right\}. $$
 {We normalize the function $\widetilde{\phi}_j^h$ by
\[
 2k\int_{\Omega_0} n(x)\widehat\phi^h_{j}(x)\overline{\widehat\phi^h_{j}(x)}\d x=1.
\]
}
Since the discretized matrices approximate the exact ones when $h\rightarrow 0$, for sufficiently small $h>0$, $J^h=J$ and
\[
\left|\widehat{\beta}_j^h-\widehat{\beta}_j\right|=O(h^2),\quad\left\|\widehat{\phi}_ {j}^h-\widehat{\phi}_ {j}\right\|_{H^1_\p(\Omega_0)}=O(h)
\]
hold for all $j\in J$. For details of the error estimation we refer to \cite{Kolata1978,Bermu2000,Engst2014}. 
 {When Assumption \ref{asp3} no longer holds, we refer to the Appendix for details.}

 {Using} \eqref{eq:orth_relation} we also get the parameter $\lambda_j^h$. 
Let $\phi^h_{j}(x):=e^{\i \widehat{\beta}_j^h x_1}\widehat{\phi}^h_{j}(x)$, then the system $\Big\{(\lambda^h_{j},\phi^h_{j})\Big\}$ for fixed $j\in J$ is the numerical approximation of \eqref{eq:orth_sys} with the following convergence result:
\begin{equation}
\label{eq:err_eig}
\left|\lambda^h_{j}-\lambda_{j}\right|=O(h^2),\quad \left\|{\phi}_{j}^h-{\phi}_{j}\right\|_{H^1(\Omega_0)}=O(h).
\end{equation}
 With these values and functions, we get the function $\theta^h(x;\psi)$ by replacing $\lambda_j$, $\phi_j$ by $\lambda_j^h$, $\phi_j^h$.

\vspace{0.3cm}

Now we are prepared to discretize the system \eqref{eq:MD2_1}-\eqref{eq:MD2_2} with $\theta(x;\psi)$  { replaced} by $\theta^h(x;\psi)$,  {and} $f_0(x)$  {replaced} by $f_0^h(x)$. The related solutions are denoted by $v_h(\beta,x)$ and $u^{(1)}_h(x)$.  {With this substitution, $u$ reads:}
\begin{equation}
\label{eq:decomp_h}
 u(x)=u^{(1)}_h(x)+u^{(2)}_h(x)
\end{equation}
where $u^{( {1})}_h$ solves \[
 \Delta u^{( {1})}_h+k^2 n u^{( {1})}_h=\sum_{j\in J}\frac{\i}{|\lambda_j^h|}\int_{\Omega_0}f(x)\overline{\phi_j^h(x)}\d x g_j^h(x)
\]
and $g_j^h$ is defined in the same way as $g_j$  { replacing }$\phi_j$ by $\phi_j^h$.
Let $g$ be a complex valued smooth function defined in $[0,1]$ such that $g(t)=\Re(g(t))+\i\sigma$ and
\[
 \Re(g(0))=-\pi\text{ and }\Re(g(1))=\pi,\quad g^{(m)}(0)=g^{(m)}(1)=0\quad\text{for all }\,m=1,2,\dots.
\]
Replace $\beta+\i\sigma$ by $g(t)$ in \eqref{eq:MD2_1}-\eqref{eq:MD2_2}  {and $\theta(x;\psi)$  by $\theta^h(x;\psi)$},  and define $\widetilde{v}(t,x):=v^h(g(t),x)g'(t)$, then $\widetilde{v}\in C^\infty_\p\left([0,1];H^2_\p(\Omega_0)\right)$. We arrive at the variational form for $\widetilde{v}$, i.e.,  { find} $\widetilde{v}\in L^2\left((0,1);H^1_\p(\Omega_0)\right)$ such that
\begin{eqnarray}\label{eq:MD2_1new}
    &&\begin{aligned}
    \int_0^1 \left<A(g(t),k)\widetilde{v}(t,\cdot),\psi(t,\cdot)\right>\d t-k^2\int_{\Omega_0}u^{(1)}_h(x)\overline{\theta^h(x;\psi)}\d x\\
    =-\int_0^1\int_{\Omega_0}e^{-\i g(t)x_1}f_0^h(x)\overline{\psi(t,x)}g'(t)\d x\d t;
    \end{aligned}\\\label{eq:MD2_2new}
    && u^{(1)}_h(x)=\int_0^1 e^{\i g(t) x_1}\widetilde{v}(t,x)\d t.
\end{eqnarray}    
Here  $f_0^h$ is defined in \eqref{eq:f0} by replacing $\phi_{j}$ by $\phi^h_{j}$.

We use the same discretization as in the CCI method and consider the discretized problem in the space $\widetilde{X}_{N,h}$, where $\widetilde{v}_{N,h}\subset \widetilde{X}_{N,h}$ is given by:
\[
\widetilde{v}_{N,h}=\sum_{\ell=1}^N\sum_{j=1}^{M}\widehat{v}_{\ell,j}\xi_\ell(t)\zeta_j(x).
\]
Then
\begin{equation}
\label{eq:MD2_2fem}
u^{(1)}_{N,h}=\sum_{j=1}^M \widehat{u}_j \zeta_j(x)\text{ where }\widehat{u}_j=\frac{1}{N}\sum_{\ell=1}^N e^{\i g(t_\ell)x_1}\widehat{v}_{\ell,j}
\end{equation}

With the test function $\psi(t,x)=\xi_{\ell'}(t)\zeta_{j'}(x)$, we get the discretized form of the system \eqref{eq:MD2_1new}-\eqref{eq:MD2_2new}:
\begin{equation}
\label{eq:MD2_1fem}
\begin{aligned}
\frac{1}{N}\delta_{\ell,\ell'}\sum_{j=1}^M\widehat{v}_{\ell,j}\left<A(t_\ell,k)\zeta_j,\zeta_{j'}\right>-k^2\sum_{j=1}^M \widehat{u}_j\left(\int_{\Omega_0}\zeta_j(x)\overline{\theta(x;\xi_{\ell'}(t)\zeta_{j'}(x))}\d x\right)\\
=-\frac{1}{N}\left(\int_{\Omega_0}e^{-\i g(t_{\ell'})x_1}f_0^h(x)\overline{\zeta_{j'}(x)}\d x\right)g'(t_{\ell'}).
\end{aligned}
\end{equation}

 {As in the CCI method, we approximate $u^{(1)}_h$ by $u^{(1)}_{N,h}$ with  coefficients defined by \eqref{eq:MD2_2fem}.}
Finally the system \eqref{eq:MD2_2fem}-\eqref{eq:MD2_1fem} is summarized  {by the following system}:
\begin{equation}
 \label{eq:MD2_matrix}
 \left(
 \begin{matrix}
  \widetilde{A}_1 & 0 & \cdots &0  & \widetilde{C}_1\\
  0 & \widetilde{A}_2 & \cdots &0 & \widetilde{C}_2\\
  \vdots & \vdots & \ddots & \vdots &  \vdots \\
  0 & 0 & \cdots & \widetilde{A}_{N}&   \widetilde{C}_{N}\\
  \widetilde{B}_1 & \widetilde{B}_2 & \cdots& \widetilde{B}_{N}  & I
 \end{matrix}
 \right)
  \left(
 \begin{matrix}
  V_1 \\ V_2 \\ \vdots \\ V_{N}\\ U
 \end{matrix}
 \right)= \left(
 \begin{matrix}
  \widetilde{F}_1 \\ \widetilde{F}_2 \\ \vdots \\ \widetilde{F}_{N} \\ 0
 \end{matrix}
 \right),
\end{equation}
where $\widetilde{A}_j$ comes from the first term of \eqref{eq:MD2_1fem}, $\widetilde{C}_j$ comes from the second term of \eqref{eq:MD2_1fem},  $\widetilde{B}_j$ comes from \eqref{eq:MD2_2fem}, $\widetilde{F}_j$ comes from the right hand side.
The size of this system is $(N+1)M\times (N+1)M$.

\subsection{Error estimations}

In the last part of this section, we present the error estimations for  {both algorithms}. Since both  solutions { $\widetilde{w}(t,x)$ of \eqref{eq:MD1_1t}-\eqref{eq:MD1_2t}}  and the solution $\widetilde{v}(t,x)$ of  \eqref{eq:MD2_1new}-\eqref{eq:MD2_2new} lie in the space $ C^\infty_\p\left([0,1];H^2_\p(\Omega_0)\right)$, we can use the error estimation  {given} in \cite{Zhang2017e}.  {The first} result is the error estimation of the interpolation in the finite dimensional space $\widetilde{X}_{N,h}$.  { Combining equation (32)   in \cite{Zhang2017e} and (46) in \cite{Lechl2016a}, we have the following results.}

\begin{theorem}
 Suppose $v\in C^\infty_\p\left([0,1];H^2_\p(\Omega_0)\right)$ and $v_{N,h}$ is the interpolation of $v$ in the subspace $\widetilde{X}_{N,h}$. Then 
 \[
  \min_{v_{N,h}\in \widetilde{X}_{N,h}}\|v-v_{N,h}\|_{L^2\left([0,1];H^1_\p(\Omega_0)\right)}\leq C\left(N^{-n}+h\right)\|v\|_{C^\infty_\p\left([0,1];H^2_\p(\Omega_0)\right)},
 \]
where $n$ can be any positive integer and $C$  {depending on $\|v\|_{C^n([0,1];H^2_\p(\Omega_0))}$} .
\end{theorem}

This implies that the interpolation decays super algebraically with respect to the parameter $1/N$ but linearly with respect to $h$. With this result, we get the error estimations for finite element solutions of the variational problems \eqref{eq:MD1_1t}-\eqref{eq:MD1_2t} and \eqref{eq:MD2_1new}-\eqref{eq:MD2_2new}. 

\begin{theorem}[Theorem 4, \cite{Zhang2017e}]
\label{th:err}
Let $\widetilde{w}$ be the exact solution of \eqref{eq:MD1_1t}-\eqref{eq:MD1_2t} and $\widetilde{w}_{N,h}$ be the finite element solution of the corresponding discretized form \eqref{eq:MD1_2_fem}-\eqref{eq:MD1_1_fem}. Let $\widetilde{v}$ be the exact solution of  \eqref{eq:MD2_1new}-\eqref{eq:MD2_2new} and $\widetilde{v}_{N,h}$ be the finite element solution of \eqref{eq:MD2_2fem}-\eqref{eq:MD2_1fem}. For sufficiently small $h>0$ and sufficiently large positive integer $N$, we have the following error estimations:
\begin{eqnarray}
 \label{eq:err_MD1}
 &&\left\|\widetilde{w}-\widetilde{w}_{N,h}\right\|_{L^2([0,1]\times\Omega_0)}\leq Ch\left(N^{-n}+h\right)\|\widetilde{w}\|_{C^\infty_\p\left([0,1];H^2_\p(\Omega_0)\right)};\\
  \label{eq:err_MD2}
 &&\left\|\widetilde{v}-\widetilde{v}_{N,h}\right\|_{L^2([0,1]\times\Omega_0)}\leq Ch\left(N^{-n}+h\right)\|\widetilde{v}\|_{C^\infty_\p\left([0,1];H^2_\p(\Omega_0)\right)}
\end{eqnarray}
where $C$ is a parameter  {depending} on $n$. 
\end{theorem}

 {For both methods}, we get the original solution from \eqref{eq:MD1_2_fem}, and the error is also easily obtained:
\begin{eqnarray}
 \label{eq:err_MD1_final}
 \left\|u-u_{N,h}\right\|_{L^2(\Omega_0)}\leq Ch\left(N^{-n}+h\right)\|\widetilde{v}\|_{C^\infty_\p\left([0,1];H^2_\p(\Omega_0)\right)};\\
  \label{eq:err_MD2_final1}
 \left\|u^{(1)}_h-u^{(1)}_{N,h}\right\|_{L^2(\Omega_0)}\leq Ch\left(N^{-n}+h\right)\|\widetilde{v}\|_{C^\infty_\p\left([0,1];H^2_\p(\Omega_0)\right)}.
\end{eqnarray}
For the CCI method, the final error estimation is already obtained by \eqref{eq:err_MD1_final}. Then in the following, we mainly  {focus} on the decomposition method. 

Recall \eqref{eq:err_eig}, we have the error bound $\left\|u^{(1)}-u^{(1)}_h\right\|_{L^2(\Omega_0)}\leq O(h)$, which converges only linearly with respect to $h$. However, recall \eqref{eq:decomp_h},
\[
 u(x)=u^{(1)}_{h}(x)+u^{(2)}_h(x)
\]
and the numerical solution is given by
\[
 u_{N,h}=u^{(1)}_{N,h}(x)+u^{(2)}_h(x).
\]
Finally we get the error estimate for the decomposition method:
\begin{equation}\label{eq:err_MD2_final}
 \left\|u- {u}_{N,h}\right\|_{L^2(\Omega_0)}=\left\|u^{(1)}_h-u^{(1)}_{N,h}\right\|_{L^2(\Omega_0)}\leq Ch\left(N^{-n}+h\right)\|\widetilde{v}\|_{C^\infty_\p\left([0,1];H^2_\p(\Omega_0)\right)}.
\end{equation}

\section{Numerical examples}

In this section, we  {give} some numerical examples to show the efficiency of  {both methods}. For both cases, we apply the same finite element discretization for quasi-periodic problem \eqref{eq:per_var}, and the meshsize $h$ is chosen as $0.005,0.01,0,02,0,04$. The parameter $N$ is chosen as $4,8,16,32,64$. 

We show numerical results for two different examples. For both examples, $f$ and $q$ are the same. $f$  {is} already defined in Remark 2, and $q$ is defined by:
\begin{equation*}
	q(x)=\begin{cases}
		0,\quad |x-b_0|>0.15;\\
		2,\quad 0.1<|x-b_0|<0.15;\\
		2\zeta(|x-b_0|;0.1,0.15),\quad\text{otherwise.}
	\end{cases}
\end{equation*}
where $b_0=(0.2,0.2)^\top$. In Example 1, the periodic refractive index $n=n_1$ is defined also in Remark 2, while in Example 2, 
\[
 n(x)=n_2(x)=3+\sin(4\pi x_1).
\]
The wave number is chosen as $\sqrt{17}$ in Example 1 and $\sqrt{12}$ in Example 2. We plot the dispersion diagrams for  {both} examples in Figure \ref{fig:dg}. From the diagrams we get the set of  {exceptional values:} For Example 1,
\[
 S(k)=\big\{-0.9577,\,0.9577\big\} \text{ and }S_-(k)=\big\{-0.9577\big\},\,S_+(k)=\big\{0.9577\big\};
\]
while for Example 2,
\[
 S(k)=\big\{-1.0326,\,1.0326\big\}\text{ and }S_-(k)=\big\{1.0326\big\},\,S_+(k)=\big\{-1.0326\big\}.
\]
Note that since the above results are numerical, they are not exact. Based on the exceptional values, we also show the integral contour $\Lambda$ in Figure \ref{fig:lambda}.

\begin{figure}[ht]
	\centering
	\begin{tabular}{c  c}
		\includegraphics[width=0.4\textwidth]{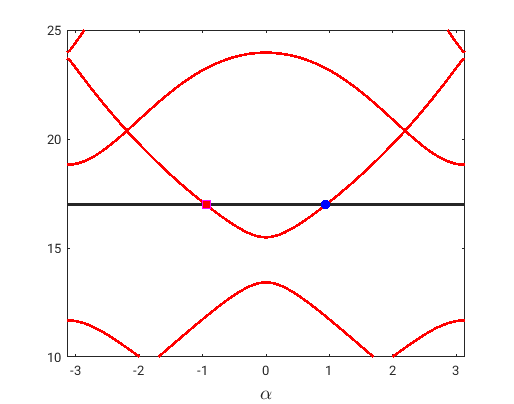} 
		& \includegraphics[width=0.4\textwidth]{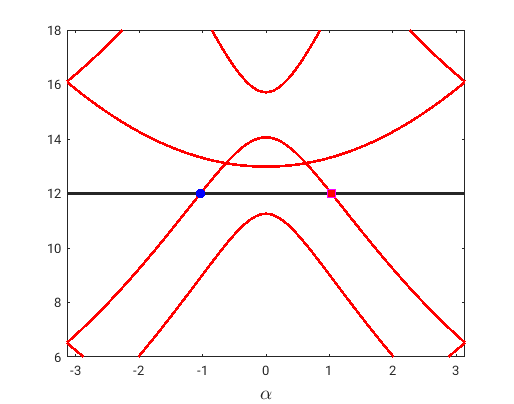}\\[-0cm]
		(a) & (b)
	\end{tabular}
	\caption{Dispersion diagrams: Example 1 in (a) and Example 2 in (b). Blue dots are points in $S_+(k)$ and purple squares are points in $S_-(k)$.}
	\label{fig:dg}
\end{figure}

\begin{figure}[ht]
	\centering
	\begin{tabular}{c  c}
		\includegraphics[width=0.4\textwidth]{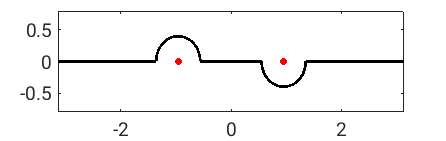} 
		& \includegraphics[width=0.4\textwidth]{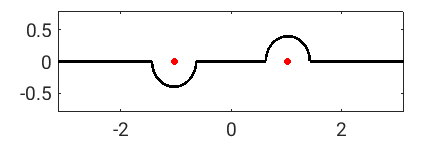}\\[-0cm]
		(a) & (b)
	\end{tabular}
	\caption{Integral contour $\Lambda$: Example 1 in (a) and Example 2 in (b). Red dots exceptional values.}
	\label{fig:lambda}
\end{figure}

\subsection{The convergence of exceptional values}

First, we show the convergence of the approximated exceptional values with respect to the meshsize $h$. As the convergence of eigenfunctions and $\lambda_{\ell,j}^h$ are  similar, they are omitted here. For $h=0.04,0.02,0.01,0.005$, we compute the exceptional values and the  {positive ones} are listed in Table \ref{table:exceptional_values}.

\begin{table}[H]
 \centering
 \begin{tabular}{|c  |c |c |c| c|}
 \hline
   & $h=0.04$ & $h=0.02$ & $h=0.01$ & $h=0.005$\\
   \hline\hline
   Example 1 & 0.8982 & 0.9435 & 0.9549 &0.9577  \\
   \hline
   Example 2 &  1.0738 & 1.0387 & 1.0337 & 1.0326  \\
   \hline
 \end{tabular}
\caption{ {Positive exceptional values }computed with different $h$'s.}
\label{table:exceptional_values}
\end{table}
If the values at $h=0.005$ is treated as ``exact'', then we plot the relative error with respect to the parameter $h$ in  logarithmic scales in the first picture in Figure \ref{fig:err}. From the plot, the slope  {of} the red curve is about $2.2$ and  {that of} the blue one is about $2.6$, which corresponds to (and even a little faster than) the convergence of the finite element method.  {Thus the convergence rates for the exceptional values are even faster than expected, i.e., $O(h^2)$.}

\subsection{Numerical results}

In this section, we focus on the numerical results obtained by the proposed methods. For both examples, we use a completely different method  {given} in \cite{Ehrhardt2009a} to  {produce ``exact solutions''}. In the computation we use Lagrangian element with meshsize $0.005$ and an extrapolation technique with data points $0.001,0.0005,0.00025$, and the solution is denoted by $u_{exa}$. Then the error is estimated as:
\[
 err_{N,h}=\frac{\|u_{N,h}-u_{exa}\|_{L^2(\Omega_0)}}{\|u_{exa}\|_{L^2(\Omega_0)}}.
\]
For the decomposition method, the parameter $\sigma$ is chosen to be $0.2$. 
For the relative errors with different examples and methods we refer to Table \ref{table:eg1_md1}-\ref{table:eg2_md2}. From the four tables, the relative errors decrease as $h$ gets smaller and $N$ gets larger, but  {the decrease stops} at the level of $3\times 10^{-3}$. This may due to the lack of accuracy of the ``exact solutions''. 

\begin{table}[H]
 \centering
\begin{tabular}{|P{2cm}|P{2cm}|P{2cm}|P{2cm}|P{2cm}|}
 \hline
  & $h=0.04$ & $h=0.02$ &$ h=0.01$ & $h=0.005$\\
  \hline
  \hline
  $N=16$ & $1.98$E$-1$ & $1.85$E$-1$ & $1.80$E$-1$ & $1.78$E$-1$\\
  \hline
  $N=32$ & $8.87$E$-2$ & $5.26$E$-2$ & $4.30$E$-2$ & $4.06$E$-2$\\
  \hline
  $N=64$ & $6.21$E$-2$ & $1.96$E$-2$ &$ 7.54$E$-3$ & $4.60$E$-3$\\
  \hline
 $ N=128$& $6.12$E$-2$ & $1.87$E$-2$ & $6.60$E$-3 $& $3.82$E$-3$\\
 \hline
 $ N=256$& $6.12$E$-2$ & $1.87$E$-2 $& $6.60$E$-3$ & $3.82$E$-3$\\
 \hline
 \end{tabular}
\caption{Relative error for Example 1, CCI method.}
\label{table:eg1_md1}
\end{table}

\begin{table}[H]
 \centering
 \begin{tabular}{|P{2cm}|P{2cm}|P{2cm}|P{2cm}|P{2cm}|}
 \hline
  & $h=0.04$ & $h=0.02$ &$ h=0.01$ & $h=0.005$\\
  \hline\hline
  $N=16$ &$7.89$E$-2$ &$2.00$E$-2$ &$5.42$E$-3$ &$3.13$E$-3$\\
  \hline
  $N=32$ &$8.27$E$-2$ &$2.05$E$-2$& $5.44$E$-3$ &$3.09$E$-3$\\
  \hline
 $ N=64$ &$8.38$E$-2$ &$2.08$E$-2$& $5.44$E$-3$ &$3.05$E$-3$\\
 \hline
 $ N=128$ &$8.36$E$-2$  &$2.08$E$-2$& $5.44$E$-3$ &$3.05$E$-3$\\
 \hline
 $ N=256$ &$8.36$E$-2$  &$2.08$E$-2$&$5.44$E$-3$ & $3.05$E$-3$\\
 \hline
 \end{tabular}
 \caption{Relative error for Example 1, decomposition method.}
\label{table:eg1_md2}
\end{table}

\begin{table}[H]
 \centering
\begin{tabular}{|P{2cm}|P{2cm}|P{2cm}|P{2cm}|P{2cm}|}
 \hline
  & $h=0.04$ & $h=0.02$ &$ h=0.01$ & $h=0.005$\\
  \hline\hline
 $ N=16 $&$6.87$E$-2 $&$4.63$E$-2$& $4.28$E$-2$& $4.22$E$-2$\\
 \hline
  $N=32 $&$4.80$E$-2 $&$1.51$E$-2$& $7.07$E$-3 $&$5.78$E$-3$\\
  \hline
  $N=64 $&$4.71$E$-2 $&$1.36$E$-2$& $4.67$E$-3$& $2.98$E$-3$\\
  \hline
  $N=128$& $4.71$E$-2$ &$1.36$E$-2$&$ 4.69$E$-3$&$ 3.02$E$-3$\\
  \hline
 $ N=256 $&$4.71$E$-2 $&$1.36$E$-2$& $4.69$E$-3$& $3.02$E$-3$\\
 \hline
 \end{tabular}
 \caption{Relative error for Example 2, CCI method.}
\label{table:eg2_md1}
\end{table}

\begin{table}[H]
 \centering
\begin{tabular}{|P{2cm}|P{2cm}|P{2cm}|P{2cm}|P{2cm}|}
 \hline
  & $h=0.04$ & $h=0.02$ &$ h=0.01$ & $h=0.005$\\
  \hline\hline
  $N=16$ &$9.01$E$-2$ &$8.04$E$-2$& $7.92$E$-2$ &$7.92$E$-2$\\
  \hline
  $N=32$ &$4.92$E$-2$ &$3.25$E$-2$& $3.05$E$-2$ &$3.03$E$-2$\\
  \hline
  $N=64$ &$3.71$E$-2$ &$1.11$E$-2$ &$4.83$E$-3 $&$3.84$E$-3$\\
  \hline
  $N=128$&$3.70$E$-2$ &$1.08$E$-2$ &$4.11$E$-3 $&$2.92$E$-3 $\\
  \hline
  $N=256$&$3.70$E$-2$ &$1.08$E$-2$ &$4.12$E$-3 $&$2.94$E$-3 $\\
  \hline
 \end{tabular}
\caption{Relative error for Example 2, decomposition method.}
\label{table:eg2_md2}
\end{table}

Now let's turn to the convergence rate. For both examples and methods, we study the dependence of the errors on $h$ and $N$ separately. To study the dependence on $h$, we fix a large $N$, i.e., $N=256$ and let the solutions with $h=0.005$ be the ``exact'' ones. Then we show the relative errors 
\[
 err_{256,h}=\frac{\|u_{256,h}-u_{256,0.005}\|_{L^2(\Omega_0)}}{\|u_{256,0.005}\|_{L^2(\Omega_0)}}\]
in Table \ref{table:FEM}. To study the dependence on $N$, we fix $h=0.005$ and let the solutions with $N=256$ be ``exact'' ones. 
Then we show the relative errors 
\[
 err_{N,0.005}=\frac{\|u_{N,0.005}-u_{256,0.005}\|_{L^2(\Omega_0)}}{\|u_{256,0.005}\|_{L^2(\Omega_0)}}.
\]
in Table \ref{table:QDR}. We also plot the data in logarithmic scales in Figure \ref{fig:err}. (b) shows the dependence on $h$ and (c) shows the dependence on $N$.  {In Figure \ref{fig:err} (b)}, the four curves are almost straight and the slopes are almost  2. This shows that the convergence rate with respect to $h$ is about $O(h^2)$. In (c), the curves are no longer close to straight ones and the slopes  {get} faster as $\log N$ gets larger. This implies that the convergence rate is super-algebraic. Both results coincide with the error estimations given in Theorem \ref{th:err}.

\begin{table}[H]
 \centering
\begin{tabular}{|P{2cm}|P{2.7cm}|P{2.7cm}|P{2.7cm}|P{2.7cm}|}
 \hline
  & Example 1 (CCI) & Example 1 (D) & Example 2 (CCI)& Example 2 (D)\\
  \hline\hline
  $h=0.04$ &$5.96$E$-2$ &$8.32$E$-2$& $4.54$E$-2$ &$3.55$E$-2$\\
  \hline
  $h=0.02$ &$1.60$E$-2$ &$2.00$E$-2$& $1.18$E$-2$ &$9.11$E$-3$\\
  \hline
  $h=0.01$ &$3.35$E$-3$ &$3.98$E$-3$ &$2.34$E$-3 $&$1.79$E$-3$\\
  \hline
 \end{tabular}
\caption{Relative errors with different $h$'s for fixed $N=256$.}
\label{table:FEM}
\end{table}

\begin{table}[H]
 \centering
\begin{tabular}{|P{2cm}|P{2.7cm}|P{2.7cm}|P{2.7cm}|P{2.7cm}|}
 \hline
  & Example 1 (CCI) & Example 1 (D) & Example 2 (CCI)& Example 2 (D)\\
  \hline\hline
  $N=16$ &$1.76$E$-1$ &$1.17$E$-4$& $4.28$E$-2$ &$8.02$E$-2$\\
  \hline
  $N=32$ &$3.87$E$-2$ &$5.52$E$-5$& $5.82$E$-3$ &$3.13$E$-2$\\
  \hline
  $N=64$ &$1.37$E$-3$ &$4.06$E$-6$ &$1.16$E$-4 $&$3.71$E$-3$\\
  \hline
  $N=128$&$1.57$E$-6$ &$4.94$E$-8$ &$4.30$E$-8 $&$4.47$E$-5 $\\
  \hline
 \end{tabular}
\caption{Relative errors with different $N$'s for fixed $h=0.005$.}
\label{table:QDR}
\end{table}

 {Finally,} we also show the error between the two different methods with the same parameters. Since the super-algebraic convergence of both algorithms with respect to $N$ is already shown, we fix $N=256$ and compute the relative errors for different $h$'s:
\[
 err_{h,256}=\frac{\left\|u^{CCI}_{256,h}-u^D_{256,h}\right\|_{L^2(\Omega_0)}}{\left\|u^{CCI}_{256,h}\right\|_{L^2(\Omega_0)}}
\]
where $u^{CCI}_{N,h}$ and $u^D_{N,h}$ are numerical results obtained from the CCI method and decomposition method, respectively. The relative  {errors} are shown in Table \ref{table:eg_compare}. The 
data are also plotted in logarithmic scales in (d) in Figure \ref{fig:err}. For Example 1, the slope of the curve is about $1.8$ and for Example 2, the slope is about $1.9$. The results also  {correspond} to the error analysis of the finite element method and the convergence of exceptional values shown in (a) in Figure \ref{fig:err}. Moreover, since the solutions for both methods coincide with each other, we can trust  {both algorithms}.

\begin{table}[H]
 \centering
\begin{tabular}{|P{2cm}|P{2cm}|P{2cm}|P{2cm}|P{2cm}|}
 \hline
  & $h=0.04$ & $h=0.02$ &$ h=0.01$ & $h=0.005$\\
  \hline\hline
  Example 1 &$5.08$E$-2$ &$1.76$E$-2$& $4.80$E$-3$ &$1.25$E$-3$\\
  \hline
  Example 2 &$1.25$E$-2$ &$3.44$E$-3$& $8.68$E$-4$ &$2.20$E$-4$\\
  \hline
 \end{tabular}
\caption{ {Relative error between two methods.}}
\label{table:eg_compare}
\end{table}

\begin{figure}[H]
	\centering
	\begin{tabular}{c   c}\includegraphics[width=0.4\textwidth]{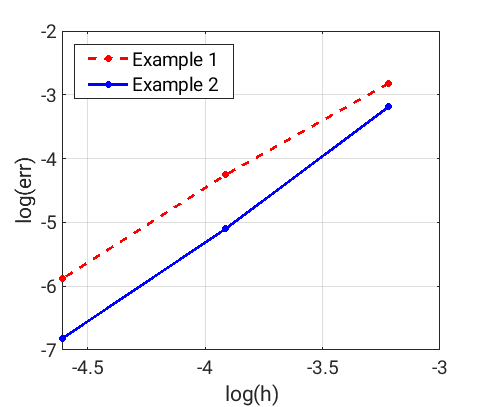} &
		\includegraphics[width=0.4\textwidth]{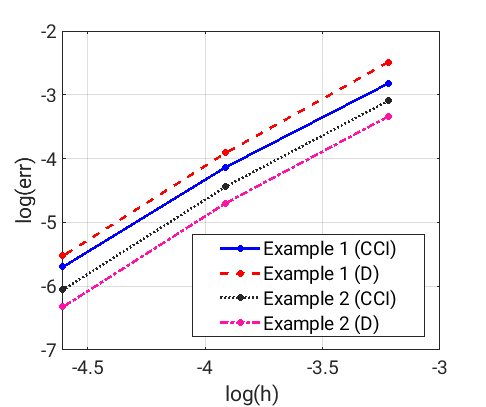} \\
		(a)&(b)\\
		 \includegraphics[width=0.4\textwidth]{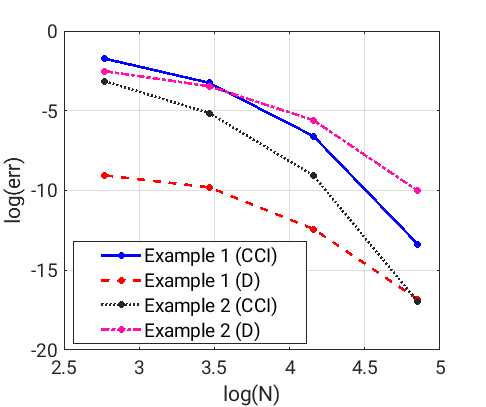}
		& \includegraphics[width=0.4\textwidth]{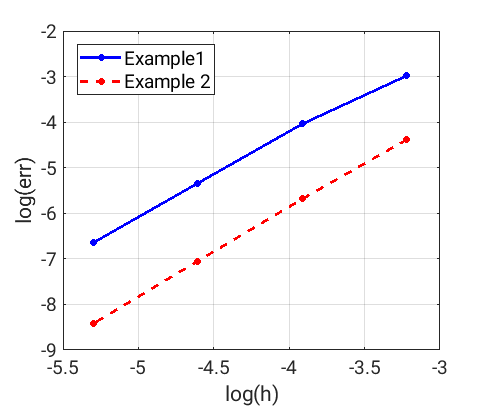}\\
		(c) &(d)
	\end{tabular}
	\caption{Convergence of exceptional values (a), relative errors depend on $h$ (b) and $N$ (c).}
	\label{fig:err}
\end{figure}

\section{Conclusion}

In the last section of this paper, we compare the two algorithms. Both methods have second order convergence  with respect to the meshsize and  {converge} super  {algebraically} with respect to the number of nodal points on $[0,1]$. The computational complexity is also similar with the same parameters.  {Both methods need to be computed by two steps. In the first step, the exceptional values and eigenfunctions are computed, where a normalization process is needed for the decomposition method. Note that,  {the computation of the  normalization is so small  that} it can be omitted, since $m_j$ is always a very small number. In the second step, we need to solve the linear system \eqref{eq:MD1_matrix} for the CCI method, and \eqref{eq:MD2_matrix} for the decomposition method. Both matrices are of the same type and size, so the  {computational} complexities and  {time} are also similar. 

Both methods also have their advantages and disadvantages. The CCI method requires the additional Assumption \ref{asp2} which is not necessary for the LAP process and the decomposition method. Fortunately this only excludes a discrete subset of $(0,+\infty)$. The propagating modes are not obtained from the CCI method. But for the CCI method, we don't need very accurate approximations for the exceptional values. On the other hand, the decomposition method  needs good approximations for the exceptional values, but it does not need Assumption \ref{asp2} and the propagating modes are also computed directly.  {Moreover, due to the corners on the contour in Figure \ref{fig:lambda}, the error from the complex contour discretization is expected to be larger than the decomposition method for the same $N$.} Since both algorithms  {converge} super-algebraically with respect to the parameter $N$, they are very efficient and we can choose  {either} algorithm according to different settings and requirements.}

\section*{Appendix}

\subsection{Smooth and analytic functions in Banach spaces}
First we recall definitions of smooth and analytic functions with values in Banach spaces.

\begin{definition}
 \label{def:banach}
 Suppose $F$ is a map from an open set $U\subset \C^N$ into a complex Banach space $X$. Then
 \begin{itemize}
  \item $F$ is analytic at $z_0\in U$ if there is an $R>0$ and a series $\{f_n:\,n\in\N\}\subset X$ such that 
  \[
   F(z)=\sum_{n=0}^\infty \frac{(z-z_0)^n}{n!}f_n
  \]
converges uniformly for $z\in B(z_0,R)\cap U$ where $B(z_0,R)$ is the disk with center $z_0$ and radius $R$.
\item $F$ is smooth at $z_0\in U$ if   {its Fr{\'e}chet derivative exists for any order}.
 \end{itemize}

\end{definition}

With Definition \ref{def:banach}, we can easily define spaces $C^\omega([a,b];H^s(\Omega_0))$, $C^\infty([a,b];H^s(\Omega_0))$ where the functions depend analytically or smoothly on the first variable. The space $C^\infty_\p([a,b];H^s(\Omega_0))$ is the subspace of $C^\infty([a,b];H^s(\Omega_0))$ with  {an additional periodic condition  on  the first variable. }

We  introduce the Floquet-Bloch transform. For a function $\phi\in C_0^\infty(\Omega)$, define the transform
\[
 (\J\phi)(\alpha,x):=(2\pi)^{-1/2}\sum_{j\in\Z}\phi\left(x+\left(\begin{matrix}
                                             j\\0
                                            \end{matrix}
 \right) \right)e^{-\i\alpha (x_1+j)}.
\]
It is easily checked that with fixed $\alpha\in\R$, $(\J\phi)(\alpha,\cdot)$ is $1$-periodic in $x_1$-direction. For fixed $x\in\Omega_0$, $e^{\i\alpha x_1}(\J\phi)(\alpha,x)$ is $2\pi$-periodic in $\alpha$. 
The properties of the transform  {are recalled} in the following theorem. For details we refer to \cite{Kuchm1993,Lechl2015e}.
\begin{theorem}
 \label{th:FBT}
 The Floquet-Bloch transform is extended to an isometry between $H^s(\Omega)$ and $L^2((-\pi,\pi);H^s_\p(\Omega_0))$ for any $s\in\R$ and its inverse transform is:
 \[
  \left(\J^{-1}\psi\right)(x)=\frac{1}{2\pi}\int_{-\pi}^\pi\psi(\alpha,x)e^{\i\alpha x_1}\d\alpha.
 \]
Moreover, the adjoint operator of $\J$ with respect to the inner product of $L^2((-\pi,\pi);L^2(\Omega_0))$, denoted by $\J^*$, equals to $\J^{-1}$.\\
$\J\phi$ depends analytically on $\alpha$ if and only if $\phi$ decays exponentially when $|x_1|\rightarrow\infty$. \\
Note here the subscript $\p$ is to indicate that the function is periodic  {with respect to $x_1$.}
\end{theorem}

 {\subsection{Normalization of the eigensystem}

In this subsection, the Assumption \ref{asp3} no longer holds. Then by solving \eqref{eq:gep}, we find all eigenvalues and eigenfunctions
\[
 \left\{\left(\widehat{\beta}_j^h,\widetilde{\phi}_{\ell,j}^h\right):\,j\in J^h,\,\ell=1,2,\dots,m_j \right\},
\]
where $m_j$ is a positive integer.
Then we normalize the system to get the eigenfunction $\widehat{\phi}^h_{\ell, j}$ such that it satisfies  \eqref{eq:orth_relation} and \eqref{eq:normal}.  The eigenfunctions are  {of} the form 
\begin{equation}\label{eq:expand}
	\widehat{\phi}_{\ell,j}^h(x)=\sum_{\ell'=1}^{m_j} c_{\ell,\ell'}^j \,\widetilde{\phi}_{\ell',j}^h(x)
\end{equation}
where the coefficients $c_{\ell,\ell'}^j\in\C$. Thus the problem is then to find out the approximated eigenvalues $\lambda_{\ell,j}^h$ and coefficients $c^{j,h}_{\ell,\ell'}$ for all $\ell,\ell'=1,2,\dots,m_j$. From \eqref{eq:orth_relation}, for fixed $\ell=1,2,\dots,m_j$,
\[
\sum_{\ell'=1}^{m_j}c^{j,h}_{\ell,\ell'}\left(\int_{\Omega_0}\left[-\i\frac{\partial }{\partial x_1}\widetilde{\phi}_{\ell',j}^h+\widehat{\beta}_j \widetilde{\phi}_{\ell',j}^h\right]\right)\overline{\widetilde{\phi}_{\ell'',j}^h}\d x=\lambda^h_{\ell,j}\sum_{\ell'=1}^{m_j} c^{j,h}_{\ell,\ell'}\left(k\int_{\Omega_0}\widetilde{\phi}_{\ell',j}^h\overline{\widetilde{\phi}_{\ell'',j}^h}\d x\right)
\]
holds for any $\ell''=1,2,\dots,m_j$. Let
\[
a^j_{\ell,\ell'}=\int_{\Omega_0}\left[-\i\frac{\partial }{\partial x_1}\widetilde{\phi}_{\ell',j}^h+\widehat{\beta}_j \widetilde{\phi}_{\ell',j}^h\right]\overline{\widetilde{\phi}_{\ell,j}^h}\d x;\quad b^j_{\ell,\ell'}=k\int_{\Omega_0}\widetilde{\phi}_{\ell',j}^h\overline{\widetilde{\phi}_{\ell,j}^h}\d x,
\]
then
\[
\left(\begin{matrix}
	a^j_{1,1} & a^j_{1,2} &\cdots & a^j_{1,m_j}\\
	a^j_{2,1} & a^j_{2,2} &\cdots & a^j_{2,m_j}\\
	\vdots & \vdots & \cdots & \vdots \\
	a^j_{m_j,1} & a^j_{m_j,2} &\cdots & a^j_{m_j,m_j}
\end{matrix}
\right) \left(\begin{matrix}
	c^{j,h}_{\ell,1} \\ c^{j,h}_{\ell,2} \\ \vdots \\ c^{j,h}_{\ell,m_j}
\end{matrix}
\right)= \lambda^h_{\ell,j}\left(\begin{matrix}
	b^j_{1,1} & b^j_{1,2} &\cdots & b^j_{1,m_j}\\
	b^j_{2,1} & b^j_{2,2} &\cdots & b^j_{2,m_j}\\
	\vdots & \vdots & \cdots & \vdots \\
	b^j_{m_j,1} & b^j_{m_j,2} &\cdots & b^j_{m_j,m_j}      
\end{matrix}
\right) \left(\begin{matrix}
	c^{j,h}_{\ell,1} \\ c^{j,h}_{\ell,2} \\ \vdots \\ c^{j,h}_{\ell,m_j}      
\end{matrix}
\right).
\]
By solving this generalized eigenvalue problem, we get all the coefficients $c^{j,h}_{\ell,\ell'}$ and eigenvalues $\lambda^h_{\ell,j}$. Then the function $\widehat\phi^h_{\ell,j}$ is obtained directly by \eqref{eq:expand}. Finally, we normalize the functions $\widehat\phi^h_{\ell,j}$ by
\[
2k\int_{\Omega_0} n(x)\widehat\phi^h_{\ell,j}(x)\overline{\widehat\phi^h_{\ell,j}(x)}\d x=1.
\]

}

\section*{Acknowledgements }
Funded by the Deutsche Forschungsgemeinschaft (DFG, German Research Foundation) – Project-ID 258734477 – SFB 1173. The author is grateful for Prof. Andreas Kirsch for valuable discussions and suggestions.

\bibliographystyle{plain}
\providecommand{\noopsort}[1]{}

\end{document}